\documentclass[reqno,10pt]{amsart}
\usepackage[colorlinks=true, linkcolor=blue, citecolor=blue]{hyperref}
\pdfoutput=1 

\usepackage{amssymb}
\usepackage{amsmath, graphicx, rotating}
\usepackage{color}
\usepackage{soul}
\usepackage[dvipsnames]{xcolor}

\allowdisplaybreaks

\usepackage{ifthen}
\usepackage{xkeyval}
\usepackage{todonotes}
\usepackage[T1]{fontenc}
\usepackage{lmodern}
\usepackage[english]{babel}

\usepackage{ upgreek }
\usepackage{stmaryrd}
\SetSymbolFont{stmry}{bold}{U}{stmry}{m}{n}
\usepackage{amsthm}
\usepackage{float}

\usepackage{ bbm }
\usepackage{ stmaryrd }
\usepackage{ mathrsfs }
\usepackage{ frcursive }
\usepackage{ comment }

\usepackage{pgf, tikz}
\usetikzlibrary{shapes}
\usepackage{varioref}
\usepackage{enumitem}

\definecolor{rouge}{rgb}{0.7,0.00,0.00}
\definecolor{vert}{rgb}{0.00,0.5,0.00}
\definecolor{bleu}{rgb}{0.00,0.00,0.8}
\usepackage[margin=1in]{geometry}
\newtheorem{theorem}{Theorem}[section]
\newtheorem*{theorem*}{Theorem}
\newtheorem{lemma}[theorem]{Lemma}
\newtheorem{definition}[theorem]{Definition}

\newtheorem{proposition}[theorem]{Proposition}

\labelformat{hypothesis}{\textbf{M\kern-0.1mm#1}}

\newtheorem{condition}{Condition}

\newcommand{\I}{\mathbf 1}

\labelformat{conditionA}{\textbf{A\kern-0.1mm#1}}
\newcommand{\R}{\mathbb R}

\theoremstyle{definition}

\newtheorem{remark}[theorem]{Remark}

\numberwithin{equation}{section}

\def\bb#1{\mathbb{#1}}

\def\scr#1{\mathscr{#1}}

\def\RR{\mathbb{R}}
\def\PP{\mathbb{P}}
\def\E{\mathbb{E}}
\def\NN{\mathbb{N}}

\def\wt{\widetilde}

\def\cR{{\mathcal R}}

\def\de{{\delta}}
\def\la{{\lambda}}
\def\si{{\sigma}}

\def\De{{\Delta}}

\def\al{{\alpha}}

\def\ep{{\epsilon}}

\def\ka{{\kappa}}
\def\Ga{{\Gamma}}
\def\ga{{\gamma}}
\def\de{{\delta}}
\def\De{{\Delta}}

\def\si{{\sigma}}

\def\la{{\lambda}}
\def\La{{\Lambda}}
\def\vare{{\varepsilon}}

\def\th{{\theta}}
\def\EE{\mathbb{ E}}
\def\th{{\theta}}
\def\si{{\sigma}}
\def\al{{\alpha}}

\def\ps{{\psi}}

\begin{document}
\title[Langevin dynamics of McKean-Vlasov type with L\'evy noises] {Exponential contractivity and propagation of chaos for Langevin dynamics of McKean-Vlasov type with L\'evy noises}

\author{Yao Liu \qquad  Jian Wang\qquad Meng-ge Zhang}

\thanks{\emph{Y. Liu:}
School of Mathematics and Statistics, Fujian Normal University, 350117 Fuzhou, P.R. China.
  \texttt{liuyaomath@163.com}}

\thanks{\emph{J. Wang:}
School  of Mathematics and Statistics \& Key Laboratory of Analytical Mathematics and Applications (Ministry of Education) \& Fujian Provincial Key Laboratory
of Statistics and Artificial Intelligence, Fujian Normal University, 350007 Fuzhou, P.R. China. \texttt{jianwang@fjnu.edu.cn}}

\thanks{\emph{M-G. Zhang:}
School of Mathematics and Statistics, Fujian Normal University, 350117 Fuzhou, P.R. China.
  \texttt{mgzhangmath@163.com}}

\maketitle

\begin{abstract} By the probabilistic coupling approach which combines a new refined basic coupling with the synchronous coupling for L\'evy processes, we obtain explicit exponential contraction rates in terms of the standard $L^1$-Wasserstein distance for the following  Langevin dynamic $(X_t,Y_t)_{t\ge0}$  of McKean-Vlasov type on $\R^{2d}$:
\begin{equation*}\left\{\begin{array}{l}
dX_t=Y_t\,dt,\\
dY_t=\left(b(X_t)+\displaystyle\int_{\R^d}\tilde{b}(X_t,z)\,\mu^X_t(dz)-\ga Y_t\right)\,dt+dL_t,\quad \mu^X_t={\rm Law}(X_t),\end{array}\right.\\
\end{equation*}
where $\ga>0$, $b:\R^d\rightarrow\R^d$ and $\tilde{b}:\R^{2d}\rightarrow\R^d$ are two globally Lipschitz continuous functions, and $(L_t)_{t\ge0}$ is an $\R^d$-valued pure jump L\'evy process. The proof is also based on a novel distance function, which is designed according to the distance of the marginals associated with the constructed coupling process. Furthermore, by applying the coupling technique above with some modifications, we also provide the propagation of chaos uniformly in time
for the corresponding mean-field interacting particle systems with L\'evy noises in the standard $L^1$-Wasserstein distance as well as with explicit bounds.

\medskip


\medskip

\end{abstract}
\allowdisplaybreaks

\section{Introduction and main results} \label{Sec Main Results}
In this paper, we consider long-time behaviors of the Langevin dynamic $(X_t,Y_t)_{t\ge0}$ of McKean-Vlasov type on $\R^{2d}$ determined by the following stochastic differential equation (SDE)
\begin{equation}\left\{\begin{array}{l}\label{Equation}
dX_t=Y_t\,dt,\\
dY_t=\left(b(X_t)+\displaystyle\int_{\R^d}\tilde{b}(X_t,z)\,\mu^X_t(dz)-\ga Y_t\right)\,dt+dL_t,\quad \mu^X_t={\rm Law}(X_t),\end{array}\right.\\
\end{equation}
where $\ga>0$, $b:\R^d\rightarrow\R^d$ and $\tilde{b}:\R^{2d}\rightarrow\R^d$ are two globally Lipschitz continuous functions, and $(L_t)_{t\ge0}$ is an $\R^d$-valued pure jump L\'evy process so that its L\'evy measure $\nu$ on $(\R^d,\scr{B}(\R^d))$ satisfies $\nu(\{0\})=0$ and $\displaystyle\int_{\R^d}(|z|\wedge|z|^2)\,\nu(dz)<\infty$.
In the field of statistical  physics, the functions $b$ and $\tilde{b}$ denote the external force and the interaction force, respectively. In particular, when $\tilde{b}\equiv0$, the solution of \eqref{Equation} corresponds to the classical Langevin dynamic, and it can be interpreted as a particle at time $t$ having a position $X_t$ and a velocity $Y_t$, which moves according to the external force.
Besides long time behaviors of \eqref{Equation}, we also focus on the mean-field interacting particle system corresponding to \eqref{Equation} which is defined, for all $N\in\NN$, by
\begin{equation}\left\{\begin{array}{l}\label{meanfield}
d\bar{X}^{i,N}_t=\bar{Y}^{i,N}_t\,dt,\\
d\bar{Y}^{i,N}_t=\Big(b(\bar{X}^{i,N}_t)+\displaystyle\frac 1 N\sum_{j=1}^N\tilde{b}(\bar{X}^{i,N}_t,\bar{X}^{j,N}_t)-\ga \bar{Y}^{i,N}_t\Big)\,dt+dL_t^i,\quad i=1,2,\cdots,N,\end{array}\right.
\end{equation}
where $\{(L_t^i)_{t\ge0}\}_{1\le i \le N}$ are independent
copies of $(L_t)_{t\ge0}$. The connection between \eqref{Equation} and \eqref{meanfield} is that the McKean-Vlasov type SDE \eqref{Equation} describes the dynamics of one particle of the interacting particle system \eqref{meanfield} when the number of particles $N$ tends to infinity. The property is closely related to the so-called $propagation~of~chaos$, which was originally studied by Kac \cite{KM} for the Boltzmann equation and was further developed by Sznitman \cite{SA} when the driven stochastic noise is a Brownian motion.

Actually, when the driven noise $(L_t)_{t\ge0}$ is  a standard Brownian motion $(B_t)_{t\ge0}$, the exponential contractivity for \eqref{Equation} recently has been studied in \cite{SK} by the coupling approach partly motivated from
 \cite{EGZ}, where (roughly speaking) a synchronous coupling is used when the distance between the marginals of the coupling process is large and  a reflection coupling is applied when the associated distance is small. In this work, we aim to develop such kind coupling idea for the Langevin dynamics of McKean-Vlasov type with L\'evy noises. In particular, in order to consider more general L\'evy driven noises we will make use of the refined basic coupling introduced in \cite{LW}. Moreover,
  due to the non-local property of the infinitesimal generator of the L\'evy process, we need to adjust the transition from the synchronous coupling for large distances to the refined basic coupling for small distances in order to suit a proper underlying distance function, which is a crucial point in our approach.
On the other hand, the propagation of chaos has been extensively studied in the Brownian setting, see e.g. \cite{CGM,DEGZ,MS,SK} and the references therein. In particular, for Langevin diffusion of McKean-Vlasov type, \cite{SK} provides bounds uniform in time for the propagation of chaos.
In this paper, we are also interested in establishing conditions on $b$ and $\tilde{b}$ such that, for any $t\ge0$, when $N\to\infty$ the law of the mean-field interacting particles driven by \eqref{meanfield} converges to the law of the Langevin dynamic $(X_t,Y_t)$ of McKean-Vlasov type with L\'evy noises given by \eqref{Equation}.

\subsection{Main results}

In order to state our main results, we first present the assumptions taken throughout the paper.

\noindent \textbf{Assumption (A0)} {\it
There
exist a constant $\ka>0$, and a non-decreasing and concave function $\si:=\sigma_{\ka,\gamma}\in C([0,\infty))\cap C^2((0,\infty))$ such that $\int_{0+}\sigma(s)^{-1}\,ds<\infty$ and, for all $r>0,$
\begin{equation}\label{f}
\si(r)\le \frac{1}{2r}J(\gamma(\ka\wedge r))(\ka\wedge r)^2,
\end{equation}
where $\gamma>0$ is given in \eqref{Equation}, and
$$
J(s):=\inf_{|x|\le s}\big(\nu\wedge(\delta_x\ast\nu)\big)(\R^d),\quad s>0.
$$
}

\noindent \textbf{Assumption (A1)}\,{\it  The function $b:\R^d\rightarrow\R^d$ is globally Lipschitz continuous, i.e., there is a constant $L_b>0$ such that for all $x, x'\in\R^d$,
$$
|b(x)-b(x')|\le L_b|x-x'|.
$$
Moreover, there are constants $R_0, \th>0$ so that for all $x, x'\in\R^d$ with $|x-x'|>R_0$,
$$
\langle b(x)-b(x'),x-x'\rangle\le-\th|x-x'|^2.
$$}

\noindent \textbf{Assumption (A2)} \, {\it The function $\tilde{b}:\R^{2d}\rightarrow\R^d$ is globally Lipschitz continuous, i.e., there exists a constant $L_{\tilde b}>0$ such that for any $x$, $x'$, $z$ and $z'\in\R^d$,
$$
|\tilde{b}(x,z)-\tilde{b}(x',z')|\le L_{\tilde b} (|x-x'|+|z-z'|).
$$}

It is easy to see that, under Assumptions \textbf{(A1)} and \textbf{(A2)} (as well as the condition that the L\'evy measure $\nu$ of the L\'evy process $(L_t)_{t\ge0}$ satisfies $\displaystyle\int_{\R^d}(|z|\wedge|z|^2)\,\nu(dz)<\infty$), the SDE \eqref{Equation} has a unique strong solution $(X_t,Y_t)_{t\ge0}$ provided that the initial distribution has finite first order moment; see Proposition \ref{prp-st} below for the details. Furthermore, it is also obvious to see that  under Assumptions \textbf{(A1)} and \textbf{(A2)} the SDE \eqref{meanfield} has a unique strong solution $ \{(\bar{X}^{i,N}_t,\bar{Y}^{i,N}_t)_{t\ge0}\}_{1\le i \le N}$, which enjoys the Markov property.

In the following, for the metric space $\bb{M}$, which is here either $\R^{2d}$ or $\R^{2Nd}$, we set
$
\scr{P}_1(\bb{M}):=\{\mu\in\scr{P}(\bb{M}):\mu(|\cdot|)<\infty\},
$ where $\scr{P}(\bb{M})$ is the class of all probability measures on $\bb{M}$.
The $L^1$-Wasserstein distance with respect to a distance function $d:\bb{M}\times\bb{M}\to\R$ is defined by
$$
\bb W_d(\mu_1,\mu_2):=\inf_{\Pi\in\mathscr{C}(\mu_1,\mu_2)}\int_{\bb{M}\times\bb{M}}d(x,y)\,\Pi(dx,dy),
$$
where the infimum is taken over all couplings $\Pi$ with marginal measures $\mu_1$ and $\mu_2$ respectively. In particular, when $d(x,y)=|x-y|$, $\bb W_d$ is just the standard $L^1$-Wasserstein distance, which is simply denoted by $\bb W_1$ below.

First, we have the following statement
about the exponential contraction of the $L^1$-Wasserstein distance for the McKean-Vlasov type Langevin dynamic \eqref{Equation} with L\'evy noises.

\begin{theorem}\label{Main theorem}
Suppose that Assumptions {\rm\textbf{(A0)}}--{\rm\textbf{(A2)}} hold, and the constants involved in Assumption {\rm\textbf{(A1)}} satisfy
\begin{equation}\label{hh}
L^2_b\ga^{-2}<\frac{3}{4}\th.
\end{equation} Then there exists a constant $C_{\tilde b}>0$ such that, when
the constant in Assumption  {\rm\textbf{(A2)}} fulfills $L_{\tilde b}\le C_{\tilde b}$,
there are constants $\lambda,\,C_1>0$ so that for any $t>0$ and
$\mu,\,\bar\mu\in\scr{P}_1(\R^{2d})$
\begin{equation}\label{W1}
\bb W_1(\mu_t,\bar\mu_t)\le C_1e^{-\lambda t}\bb W_1(\mu,\bar\mu),
\end{equation}
where $\mu_t$ and $\bar \mu_t$ are the laws of the solutions $(X_t,Y_t)_{t\ge0}$ to \eqref{Equation} with initial distribution
$\mu$ and
$\bar \mu$ respectively.
\end{theorem}

\begin{remark}\label{Remark}
\begin{itemize}
\item[{\rm(i)}] To obtain Theorem \ref{Main theorem}, we in fact establish the following assertion
\begin{equation}\label{W}
\bb W_{\rho}(\mu_t,\bar\mu_t)\le e^{-\lambda t}\bb W_{\rho}(\mu,\bar \mu),
\end{equation}
where $\bb W_{\rho}$ denotes the $L^1$-Wasserstein
distance with respect to a designed distance $\rho$ on $\R^{2d}$ (see \eqref{Metric} below) that is equivalent to the Euclidean distance. Moreover, according to the proof of Theorem \ref{Main theorem}, we have the following explicit estimates for the constant $C_{\tilde b}$ in the statement as well as
the constants $\lambda$ and $C_1$ involved in the exponential contraction
\eqref{W1}:
\begin{equation}\label{n}
L_{\tilde b}\le C_{\tilde b}:= \min\left\{\frac{L_b}{4}e^{-\La_1},\frac{1}{8\sqrt{2}}\ep\ga^2\tau(1-2\tau)e^{-\La_2}\right\},
\end{equation}
\begin{equation}\label{La}
\lambda=\min\bigg\{\al\ga e^{-\La_1},\al\ga\frac{k_0-4}{4(k_0\al+1)}e^{-\La_1},\frac{\tau\ga\ep\vare}{4}e^{-\La_2}\bigg\}
\end{equation}
and
\begin{equation}\label{C1}
C_1=\frac{\sqrt{2}\max\{\al+1,\ga^{-1}\}e^{\La_2}}{\ep\min\big\{(1-2\tau)/(2\sqrt{2}),\ga^{-1}/\sqrt{3}\big\}}
\end{equation}
with
$$
\La_1:=2\al\ga(1+k_0\al)\int^{R_1}_0\si\Big(\frac{s}{1+k_0\al}\Big)^{-1}ds,\quad
\La_2:=2\al\ga(1+k_0\al)\int^{2R_1}_0\si\Big(\frac{s}{1+k_0\al}\Big)^{-1}ds,
$$
and
$$
\al:=2L_b\ga^{-2},\quad \tau:=\min\{{1}/{8}, \ga^{-2}\th-4L^2_b\ga^{-4}/3\},
$$
where $k_0>4$, $R_1$ is a positive (large) constant, and $\ep$ and $\vare$ are positive (small) constants.
In particular, we should note that the estimates above for the constants $\lambda$ and $C_1$ are independent of dimension.
\item[{\rm (ii)}]
 According to \eqref{f} and $\int_{0+}\sigma(s)^{-1}\,ds<\infty$, we know that $$\lim_{r\to0} J(r)=\lim_{r\to 0} \inf_{|x|\le r}\big(\nu\wedge(\delta_x\ast\nu)\big)(\R^d)=\infty.$$  Hence, Assumption {\rm\textbf{(A0)}} roughly indicates that there are many active small jumps for the L\'evy process $(L_t)_{t\ge0}$. Such kind assumption  can be regarded as a non-degenerate condition for the L\'evy measure $\nu$ near zero; see \cite{LMW,LW} and the references therein for more details.
     In particular, according to \cite[Example 1.2]{LW}, when the L\'evy measure $\nu$
   of the L\'evy process $(L_t)_{t\ge0}$
    satisfies
    $$\nu(dz)\ge c_0\I_{\{0<z_1\le1\}}|z|^{-d-\beta}\,dz$$ with $c_0>0$, $\beta\in (0,1)$ and $z=(z_1,z_2,\cdots,z_d)\in \R^d$, then we can take $\si(r)=c_1r^{1-\beta}$ for $r\in (0,1]$ in \eqref{f}.
    \end{itemize}
\end{remark}

Below, we provide the result on uniform in time propagation of chaos of the mean-field interacting particle system with L\'evy jumps \eqref{meanfield} towards the Langevin dynamic $(X_t,Y_t)$ of McKean-Vlasov type with L\'evy noises given by \eqref{Equation}. For this, we define the normalized $l^1$-distance
\begin{equation}\label{lN}
l^1_N((x,y),(x',y')):=N^{-1}\sum_{i=1}^N
|(x^i,y^i)-(x'^i,y'^i)|, \quad  ((x,y),(x',y'))\in\RR^{2Nd}\times\RR^{2Nd},
\end{equation}
where
$x=(x^1,x^2,\cdots, x^N)\in \RR^{Nd}$, and
$|\cdot|$ is the Euclidean distance on $\R^{2d}$. Besides, we also need the following assumption.

\noindent \textbf{Assumption (A3)} \, {\it  The L\'evy measure $\nu$ of the L\'evy process $(L_t)_{t\ge0}$ satisfies $$\displaystyle\int_{\R^d}|z|^2\,\nu(dz)<\infty.$$}

\begin{theorem}\label{chaos}
Suppose that Assumptions {\rm\textbf{(A0)}}--{\rm\textbf{(A3)}}, \eqref{hh} and \eqref{n} hold. Let
$\mu$ and
$\bar{\mu}$ be two probability measures on $\RR^{2d}$ with finite second order moment. Then there are constants
$C_1,C_2>0$ so that for all $t>0$ and $N\ge1$,
\begin{equation}\label{WN1}
\bb W_{l^1_N}(\mu^{\otimes N}_t,\bar\mu^N_t)\le C_1e^{-\lambda t}\bb W_{l^1_N}(\mu^{\otimes N},\bar\mu^{\otimes N})+C_2N^{-1/2},\quad t>0,
\end{equation}
where $\mu^{\otimes N}_t$ is the product law of $N$ independent solutions to \eqref{Equation} with initial distribution
$\mu$, and $\bar\mu^N_t$ is the law of the particles driven by \eqref{meanfield} with initial distribution
$\bar{\mu}^{\otimes N}$.
\end{theorem}

\begin{remark}\label{Remark2}
\begin{itemize}
\item[{\rm (i)}]
The constants $\la$ and $C_1$ in Theorem \ref{chaos} are given by \eqref{La} and \eqref{C1} respectively, and $C_2$ is independent of $t$ and $N$.
On the other hand, similar to Theorem \ref{Main theorem}, in order to prove Theorem \ref{chaos} we turn to establish the following assertion
\begin{equation}\label{WN}
\bb W_{\rho_N}(\mu^{\otimes N}_t,\bar\mu^N_t)\le e^{-\lambda t}\bb W_{\rho_N}(\mu^{\otimes N},\bar\mu^{\otimes N})+C_0N^{-1/2},\quad t>0,
\end{equation}
where $\bb W_{\rho_N}$ denotes the $L^1$-Wasserstein distance with respect to the metric $\rho_N$ on $\R^{2Nd}$ that is equivalent to $l^1_N$ (see \eqref{NMetric} and \eqref{equivalenceN} below for the details).
\item[{\rm (ii)}]
The proof of Theorem \ref{chaos} will partly follow that of Theorem \ref{Main theorem}. The main difference here is that we need to quantify the drift terms  $\int_{\R^d}\tilde{b}(X_t,z)\,\mu^X_t(dz)$ and $N^{-1}\sum_{j=1}^N\tilde{b}(\bar{X}^{i,N}_t,\bar{X}^{j,N}_t)$ involved in the SDEs \eqref{Equation} and  \eqref{meanfield} respectively. For this purpose, we require that the associated processes, and so the L\'evy measure $\nu$, have
finite second order moment.
\end{itemize}
 \end{remark}

\subsection{Overview on the existing results and approaches} Lots of papers are devoted to the existence and the uniqueness of McKean-Vlasov processes; see \cite{MS} and references therein for the history and some developments, and \cite{LMW} for the case that the driven noise is a L\'evy process.  As mentioned before, the equation \eqref{meanfield} can describe the motion of $N$ particles subject to damping, random collisions and a confining potential as well as the
interaction with one another through an interacting potential, while \eqref{Equation} has been paid by attentions in the probability and the PDE community since its density describes the limit dynamic of particles. Therefore, long time behaviors of McKean-Vlasov processes as well as the associated propagation of chaos have wide
interest and are the objective of many works.

When the driven noise is a Brownian motion, there are a few works on this topic. Some analytic approaches (mostly motivated by the famous work on the hypoellipticity due to Villani) were adopted to get long time behaviors for nonlinear Langevin diffusions; see \cite{BD,GLWZ}. With the aid of functional inequalities, the propagation of chaos was also proved in \cite{GM}.  Meanwhile, the probabilistic techniques have been fully used to analyse nonlinear Langevin diffusions, and, in particular, see \cite{BGM,GLM,SK} for the powerful coupling techniques. Among them, based on a novel distance function which combines with two contraction results for large and small distances and uses both synchronous coupling and reflection coupling that adjusted to the distance, global contractivity for nonlinear Langevin diffusions with general forces and the associated uniform in time propagation of chaos are recently obtained in \cite{SK}.

From the viewpoints of mathematics and physics, it is natural to consider McKean-Vlasov processes driven by L\'evy noises; see \cite{JMW2,JMW, MW}. The propagation of chaos phenomenon has also been studied for McKean-Vlasov SDEs driven by L\'evy process. For example, when the driving L\'evy noise has a finite second order moment, the quantitative convergence rates of propagation of chaos under standard Lipschitz assumptions on the drift and the diffusion coefficients were presented in \cite[Theorem 1.3]{JMW}. Recently, by the regularizing properties and the dynamics of the associated semigroup, quantitative weak propagation of chaos for McKean-Vlasov SDEs driven by a rotationally invariant $\alpha$-stable process with $\alpha\in (1,2)$ has been established in \cite{Cav}. Nevertheless, to the best of our knowledge, long time behaviors for the  McKean-Vlasov process \eqref{Equation} driven by L\'evy noises are still unknown in the literature. One of purposes of this paper is to fill in this gap.

Our approach here is heavily inspired by that of \cite{SK}. We have to handle the difficulties coming from the nonlinearity and the degenerate property of the Langevin dynamic (the latter property is analogous to the hypoellipticity of Langevin diffusions), but in the present setting we also need to take care of the (discontinuous) driven noise simultaneously. In detail, we will resort to the refined basic coupling for L\'evy noises that was first introduced in \cite{LW} and the construction of the proper metric function that fulfills the non-local property of L\'evy noises.
Such kind construction is completely different from that for Langevin diffusions known in the literature.
On the other hand, we combine these two new points above with some ideas from \cite{SK}, specially by considering two separate metrics for large and small distances that the synchronous coupling and the refined basic coupling apply respectively. Besides, we also partly make use of the coupling approach in  \cite{BW} to yield the exponential ergodicity of the following stochastic Hamiltonian system driven by L\'evy noises:
\begin{equation*}\left\{\begin{array}{l}
dX_t=(aX_t+b Y_t)\,dt,\\
dY_t=U(X_t,Y_t)\,dt+dL_t.\end{array}\right.\\
\end{equation*}

It is clear that the paper \cite{BW} is concerned on the classical Langevin dynamics with L\'evy noises; that is, $\tilde b\equiv0$ in \eqref{Equation}. In particular, the contraction is shown in \cite[Theroems 1.1 and 1.3]{BW} in a specific $L^1$- Wasserstein distance with respect to a semimetric involving a Lyapunov function. Then, it is possible to extend the approach of \cite{BW} to nonlinear setting; see \cite{GLM} for the related work on nonlinear Langevin diffusions. As mentioned in \cite{SK}, the advantage of the approach in the present paper is that the rates of the global contractivity in Theorem \ref{Main theorem} are dimension free. On the other hand, Theorem \ref{chaos} shows uniform in time propagation of chaos, as well as the explicit quantitative rates, of the particle system \eqref{meanfield} towards the Langevin dynamic of McKean-Vlasov type with L\'evy noises given by \eqref{Equation}. We would like to mention that, although the result of such kind has been proved in \cite[Theorem 1.4]{JMW} even with distribution dependent noises, the bounds obtained herein are uniform only on a finite time interval.

\ \

The paper is arranged as follows. In the next section, we first claim the SDE  \eqref{Equation} has the unique strong solution. This enables us to introduce the corresponding decoupled SDE \eqref{decoupled1}  and apply the Markov coupling technique. Then, we construct a coupling process for the decoupled SDE \eqref{decoupled1} through a coupling operator. In particular, a new refined basic coupling and the synchronous coupling for L\'evy processes are fully used. With the aid of all the preparations, the proof of Theorem \ref{Main theorem} is given in Section \ref{Proofs}. Section \ref{section4} is devoted to the proof of Theorem \ref{chaos}, which is also based on the coupling techniques. Moreover, the difference of the drift term between  \eqref{Equation} and  \eqref{meanfield} is required to quantify carefully.

\section{Strong solution of Langevin dynamics of McKean-Vlasov type with jumps and coupling process of their decoupled version}
\subsection{Uniqueness of strong solution to the SDE  \eqref{Equation}}
In this part, we will apply Theorem \ref{app} in the Appendix to verify that the SDE \eqref{Equation} has a unique strong solution $(X_t,Y_t)_{t\ge0}$ in the framework of the paper. We first note that
under Assumption \textbf{(A2)}, for any $x,x'\in \R^d$ and $\mu_1,\mu_2\in \scr{P}_1(\R^d)$,
\begin{equation}\label{e:addremark}\left|\int_{\R^d}\tilde b(x,z)\,\mu_1(dz)-\int_{\R^d}\tilde b(x',z)\,\mu_2(dz)\right|\le L_{\tilde b} (|x-x'|+\bb W_1(\mu_1,\mu_2));\end{equation} see \cite[Remark 3.4]{LMW}.

\begin{proposition}\label{prp-st}
Under Assumptions {\rm \textbf{(A1)}} and {\rm\textbf{(A2)}}, the SDE \eqref{Equation} has a unique strong solution $(X_t,Y_t)_{t\ge0}$, which also has weak well-posedness.
\end{proposition}

\begin{proof} Let $h=\left(
\begin{array}{c}
   x\\
y
   \end{array}
\right)\in \R^{2d}$ with $x,y\in \R^d$. For any $t>0$, set
$$
H_t:=
\left(
\begin{array}{c}
   X_t\\
   Y_t
   \end{array}
\right),\quad
b(H_t,\scr{L}_{H_t}):=
\left(
\begin{array}{c}
   Y_t\\
   b(X_t)+\displaystyle\int_{\R^d}\tilde{b}(X_t,z)\,\mu^X_t(dz)-\ga Y_t
   \end{array}
\right)$$ and $$
\si:=
\left(
\begin{array}{cc}
   \mathbf{0} & \mathbf{0}\\
   \mathbf{0} & \mathbf{I}\\
   \end{array}
\right),\quad
\hat{L}_t:=
\left(
\begin{array}{c}
   L_t'\\
   L_t
   \end{array}
\right),$$
where $\mu^X_t$ denotes the law of $X_t$, $(L_t')_{t\ge0}$ is a L\'evy process on $\R^d$ that is an independent version of $(L_t)_{t\ge0}$, and $\mathbf{I}$ and $\mathbf{0}$ denote the $d\times d$ unit matrix and the $d\times d$ null matrix respectively. Then, we can rewrite the SDE \eqref{Equation} as follows \begin{equation}\label{reEquation}
dH_t=b(H_t,\scr{L}_{H_t})\,dt+\si d\hat{L}_t.
\end{equation}
Now, under Assumptions \textbf{(A1)} and \textbf{(A2)}, we claim that Assumptions \textbf{(H1)} and \textbf{(H2)} in the Appendix are satisfied.

Indeed, it follows from Assumptions \textbf{(A1)} and \textbf{(A2)} that for all $h_1,h_2\in\R^{2d}$ and $\hat{\mu},\tilde{\mu}\in\scr{P}_1(\R^{2d})$,
\begin{align*}
&\langle b(h_1,\hat{\mu})-b(h_2,\tilde{\mu}),h_1-h_2\rangle\\
&=\Bigg\langle\left(
\begin{array}{c}
   y_1\\
   b(x_1)+\displaystyle \int_{\R^d}\tilde{b}(x_1,z)\,\hat{\mu}^X(dz)-\ga y_1
   \end{array}
\right)
-
\left(
\begin{array}{c}
   y_2\\
   b(x_2)+\displaystyle \int_{\R^d}\tilde{b}(x_2,z)\,\tilde{\mu}^X(dz)-\ga y_2
   \end{array}
\right),\,
\left(
\begin{array}{c}
   x_1-x_2\\
   y_1-y_2
   \end{array}
\right)
\Bigg\rangle\\
&\le |y_1-y_2|\cdot|x_1-x_2|+L_b|x_1-x_2|\cdot|y_1-y_2|+L_{\tilde{b}}\left(|x_1-x_2|+\bb{W}_1(\hat{\mu},\tilde{\mu})\right)\cdot|y_1-y_2|-\ga|y_1-y_2|^2\\
&\le (1+L_b+L_{\tilde{b}})\left(|h_1-h_2|^2+|h_1-h_2|\bb{W}_1(\hat{\mu},\tilde{\mu})\right),
\end{align*}
where in the equality above $\hat{\mu}^X$ (resp.\ $\tilde{\mu}^X$) denotes the marginal distribution of $X$ with respect to $\hat{\mu}$ (resp.\ $\tilde \mu$),
in the first inequality we used \eqref{e:addremark}
 and its proof (see \cite[Remark 3.4]{LMW}) as well as
 the fact that $\displaystyle\int_{\R^d}\tilde{b}(x_1,z)\,\hat{\mu}^X(dz)=\displaystyle\int_{\R^{2d}}\tilde{b}(x_1,z)\,\hat{\mu}(dz\,dw)$ for the marginal distribution $\hat{\mu}^X(dz)$ of $\hat{\mu}(dz\,dw)$, and the last inequality immediately  follows from the property that $\max\{ |x_1-x_2|, |y_1-y_2|\}\le |h_1-h_2|.$

On the other hand, due to Assumption \textbf{(A2)} again,  for all $\mu\in\scr{P}_1(\R^{2d})$,
\begin{align*}
|b(0,\mu)|&=
\left|\left(
\begin{array}{c}
   0\\
   b(0)+\displaystyle\int_{\R^d}\tilde{b}(0,z)\,\mu^X(dz)
   \end{array}
\right)\right| \le|b(0)|+\displaystyle\int_{\R^d}|\tilde{b}(0,z)|\,\mu^X(dz)\\
&\le|b(0)|+\displaystyle\int_{\R^d}(|\tilde{b}(0,0)|+L_{\tilde{b}}|z|)\,\mu^X(dz)\le \max\{ |b(0)|+|\tilde b(0,0)|, L_{\tilde{b}}\} (1+\mu(|\cdot|)).
\end{align*}

Therefore, according to Theorem \ref{app} in the Appendix, the SDE \eqref{reEquation}, and so the SDE \eqref{Equation}, has a unique strong solution $(X_t,Y_t)_{t\ge0}$.
\end{proof}

It is known that the solution to the distribution dependent SDE \eqref{Equation} does not enjoy the Markov property. In order to make full use of the Markov coupling approach, we have to recourse to its decoupled version (by freezing distribution). In the following, denote by $\mu_t={\rm Law}((X_t, Y_t))=:(\mu_t^X,\mu_t^Y)$ if $\mu:={\rm Law}((X_0, Y_0))$. For any $\mu\in\scr{P}_1(\R^{2d})$, we consider the decoupled SDE associated with \eqref{Equation}:
\begin{equation}\left\{\begin{array}{l}\label{decoupled1}
dX^{\mu }_t=Y^{\mu }_t\,dt,\\
dY^{\mu }_t=\left(b(X^{\mu }_t)+\displaystyle\int_{\R^d}\tilde{b}(X^{\mu }_t,z)\,\mu^{X}_{t}(dz)-\ga Y^{\mu}_t\right)\,dt+dL_t.\end{array}\right.
\end{equation}
Note that, for fixed $\mu\in\scr{P}_1(\R^{2d})$, \eqref{decoupled1} is a classical (distribution-independent) SDE with time-dependent drift. It is easy to see that, under Assumptions \textbf{(A1)} and \textbf{(A2)}, the SDE \eqref{decoupled1} has  a unique strong solution $(X^{\mu}_t,Y^{\mu}_t)_{t\ge0}$ for any initial condition; moreover, for any fixed ${\mu}\in\scr{P}_1(\R^{2d})$, the process $(X^{\mu}_t,Y^{\mu}_t)_{t\ge0}$ has the Markov property and the generator of the process $(X^{\mu}_t,Y^{\mu}_t)_{t\ge0}$, acting on $f\in C^2_b(\R^{2d})$, is given by
\begin{equation}\label{L}
\begin{split}
(\scr{L}_{[{\mu}],t}f)(x,y)
&=\big\langle y, \nabla_xf(x,y)\big\rangle+\bigg\langle b(x)+\displaystyle\int_{\R^d}\tilde{b}(x,z)\, {\mu}^X_t(dz)-\ga y, \nabla_yf(x,y)\bigg\rangle\\
&\quad+\int_{\R^d}\big(f(x,y+z)-f(x,y)-\langle\nabla_yf(x,y),z\rangle\I_{\{|z|\le1\}}\big)\,\nu(dz)\\
&=: (\scr{L}_{[{\mu}],t, 0}f)(x,y)+(\scr{L}_1 f)(x,y),
\end{split}
\end{equation} where $\nabla_x $ is the first order gradient operator with respect to the variable $x$ and the same for $\nabla_y$.
Furthermore, via the strong well-posedness (so the weak well-posedness) of \eqref{Equation}, we obtain that $\mu_{t}={\rm Law}((X^{\mu}_t,Y^{\mu}_t))$ if $\mu={\rm Law}((X^{\mu}_0,Y^{\mu}_0))$.
Therefore, in order to obtain Theorem $\ref{Main theorem}$, thanks to Remark \ref{Remark} it suffices to establish the following exponential contraction for the inhomogeneous Markov process $(X^{\mu}_t,Y^{\mu}_t)_{t\ge0}$ solving SDE \eqref{decoupled1}:
\begin{equation}\label{Wde}
\bb W_{\rho}(\mu_0 P^{\mu}_t,\bar{\mu}_0 P^{\bar{\mu}}_t)\le e^{-\la t}\bb W_{\rho}(\mu_0,\bar{\mu}_0),
\end{equation}
where
$\mu_0 P^{\mu}_t$ and
$\bar{\mu}_0 P^{\bar \mu}_t$ stand for the laws of the solutions $(X^{\mu}_t,Y^{\mu}_t)_{t\ge0}$ and $(X^{\bar \mu}_t,Y^{\bar \mu}_t)_{t\ge0}$ to the SDE \eqref{decoupled1} with initial distribution
$\mu_0$ and
$\bar{\mu}_0$ respectively, and $\bb W_{\rho}$ is the $L^1$-Wasserstein distance mentioned in Remark \ref{Remark}.

\subsection{Coupling operator for the decoupled SDE \eqref{decoupled1} } \label{sub coupling operator}

In this subsection, we will construct a coupling operator for $\scr{L}_{[{\mu}],t}$ defined by \eqref{L}. For this, we have to fully use the refined basic coupling and the synchronous coupling for pure jump L\'evy processes.
We start with a coupling for the L\'evy measure $\nu$, and introduce the notation
$$
x\rightarrow x+z, \quad \nu(dz),
$$
which means a transition from a point $x$ to the point $x+z$ with the jump intensity $\nu(dz)$.
Instead of considering the process $(X^{\mu}_t,Y^{\mu}_t)_{t\ge0}$ directly, we will study its transformation $(X^{\mu}_t,\widetilde{Y}^{\mu}_t)_{t\ge0}$ (see \cite[Section 2]{BW}):
\begin{equation}\left\{\begin{array}{l}\label{New equation}
dX^{\mu}_t=(-\ga X^{\mu}_t+\ga \widetilde{Y}^{\mu}_t)\,dt,\\
d\widetilde{Y}^{\mu}_t=\ga^{-1}\left(b(X^{\mu}_t)+\displaystyle\int_{\R^d}\tilde{b}(X^{\mu}_t,z)\,\mu^{X}_t(dz)\right)\,dt+\ga^{-1}dL_t;\end{array}\right.\\
\end{equation}
that is, for any $t>0$, $\widetilde{Y}^{\mu}_t:=X^{\mu}_t+\ga^{-1}Y^{\mu}_t$.

 Roughly speaking, the essential idea of the basic coupling is to make two marginal processes jump to the same point with the biggest possible rate, where the biggest jump rate is the maximal common part of the jump intensities. In the L$\'{e}$vy setting, for $x\in\R^d$, the maximal common part takes the form
\begin{equation}\label{a}
\nu_x(dz):=\big(\nu\wedge(\delta_x\ast\nu)\big)(dz),
\end{equation}
where $\delta_x$ is the Dirac measure or the unit mass at the point $x$.
Then, according to the strategy of \cite[Section 2.1]{LW} and the structure of the modified SDE \eqref{New equation}, \emph{the refined basic coupling} of  the pure jump L\'evy process $(L_t)_{t\ge0}$ is constructed via the following relationship: for any $x,\,x',\,y,\,y'\in\R^d$ with $q:=x-x'+\ga^{-1}(y-y')$
\begin{equation}\label{Refined basic coupling}
(y,y')\longrightarrow
\begin{cases}
(y+z,\,y'+z+\ga(q)_\ka),    &\frac{1}{2}\nu_{-\ga(q)_\ka}(dz),\\
(y+z,\,y'+z-\ga(q)_\ka),    &\frac{1}{2}\nu_{\ga(q)_\ka}(dz),\\
(y+z,\,y'+z),    &(\nu-\frac{1}{2}\nu_{-\ga(q)_\ka}-\frac{1}{2}\nu_{\ga(q)_\ka})(dz),\\
\end{cases}
\end{equation}
where $\ka>0$ is given in Assumption \textbf{(A0)}, and $\nu_{\ga(q)_\ka}$ (resp.\, $\nu_{-\ga(q)_\ka}$) is defined by \eqref{a} with $x=\ga(q)_\ka$ (resp.\, $x=-\ga(q)_\ka$). Here and in what follows, for any $x\in\R^d$, define
$$
(x)_{\ka}=\left(1\wedge\frac{\kappa}{|x|}\right)x,
$$
and we make the convention that $(x)_\kappa=0$ for $x=0$. Furthermore, the operator associated with \eqref{Refined basic coupling} can be expressed as follows: for $g\in C^2_b(\R^{2d})$,
\begin{equation}\label{Lxx'}
\begin{split}
\big(\widetilde{\scr{L}}_{x,x'}g\big)(y,y')
&:=\frac{1}{2}\int_{\R^d}\big(g(y+z,y'+z+\ga(q)_\ka)-g(y,y')-\big\langle\nabla_yg(y,y'),z\big\rangle\I_{\{|z|\le1\}}\\
&\qquad\qquad\,\,-\big\langle\nabla_{y'}g(y,y'),z+\ga(q)_\ka\big\rangle\I_{\{|z+\ga(q)_\ka|\le1\}}\big)\nu_{-\ga(q)_\ka}(dz)\\
&\quad+\frac{1}{2}\int_{\R^d}\big(g(y+z,y'+z-\ga(q)_\ka)-g(y,y')-\big\langle\nabla_yg(y,y'),z\big\rangle\I_{\{|z|\le1\}}\\
&\qquad\qquad\quad-\big\langle\nabla_{y'}g(y,y'),z-\ga(q)_\ka\big\rangle\I_{\{|z-\ga(q)_\ka|\le1\}}\big)\nu_{\ga(q)_\ka}(dz)\\
&\quad+\int_{\R^d}\big(g(y+z,y'+z)-g(y,y')-\big\langle\nabla_yg(y,y'),z\big\rangle\I_{\{|z|\le1\}}\\
&\qquad\qquad\,-\big\langle\nabla_{y'}g(y,y'),z\big\rangle\I_{\{|z|\le1\}}\big)\left(\nu-\frac{1}{2}\nu_{-\ga(q)_\ka}-\frac{1}{2}\nu_{\ga(q)_\ka}\right)(dz).
\end{split}\end{equation}

Naturally, the synchronous coupling of  the pure jump L\'evy process $(L_t)_{t\ge0}$ is given
for any $x,\,x',\,y,\,y'\in\R^d$ by
$$
(y,y')\longrightarrow
(y+z,\,y'+z),     \quad  \nu(dz),
$$
 and the associated operator is of the following form
\begin{equation}
\begin{split}\label{Lxx''}
\big(\widetilde{\scr{L}}^\ast_{x,x'}g\big)(y,y')
&:=\int_{\R^d}\big(g(y+z,y'+z)-g(y,y')-\big\langle\nabla_yg(y,y'),z\big\rangle\I_{\{|z|\le1\}}\\
&\qquad\quad\,\,\,-\big\langle\nabla_{y'}g(y,y'),z\big\rangle\I_{\{|z|\le1\}}\big)\,\nu(dz).
\end{split}\end{equation}

In the present setting we will apply the  refined basic coupling and the synchronous coupling simultaneously. More explicitly,  we only use the synchronous coupling when the distance of two marginal processes is large enough, otherwise we will apply the refined basic coupling.~To describe the regions that these two different couplings apply, we further need some more notations. For any $x$, $x'$, $y$, $y'\in\R^d$, set
$$
v:=x-x', \quad w:=y-y', \quad q:=v+\ga^{-1}w
$$
and
\begin{equation}\label{e:metric1} r_s((x,y),(x',y')):=\al|v|+|q|,  \quad r^2_l((x,y),(x',y')):=A|v|^2+B\langle v,w\rangle+C|w|^2,
\end{equation}
where $\al$, $A$, $B$ and $C$ are positive constants determined later.~Note that, with explicit
choice of these constants (see \eqref{al} and \eqref{rl} below), we will obtain that $r_s((x,y),(x',y'))$ and $r_l((x,y),(x',y'))$ are equivalent metrics on $\R^{2d}$, and both of them are equivalent to the Euclidean distance; see Lemma \ref{lemma1-app} in the Appendix. In particular, there exist constants $\ep$, $\vare\in (0,{1}/{2}]$ such that
$$
2\ep r_l((x,y),(x',y'))\le r_s((x,y),(x',y'))$$ and $$\vare r_s((x,y),(x',y'))\le r_l((x,y),(x',y'))
$$ for any $x$, $x'$, $y$, $y'\in\R^d$.
Furthermore, we define
$$
\De((x,y),(x',y')):=r_s((x,y),(x',y'))-\ep r_l((x,y),(x',y'))
$$ and
\begin{equation}\label{e:Gamma}
D_\Ga:=\sup_{((x,y),(x',y'))\in\Ga}\De((x,y),(x',y')),
\end{equation}
where $\Ga\subset\R^{4d}$ is a compact set given by
$$
\Ga:=\{
((x,y),(x',y'))\in \R^{4d}:r^2_l((x,y),(x',y'))\le\cR\}
$$
with $\cR$ to be chosen later (see \eqref{l}).

\ \

With discussions above, for any fixed $\hat{\mu},\tilde{\mu}\in\scr{P}_1(\R^{2d})$, $x,x',y,y'\in\R^d$, $t>0$ and $f\in C^2_b(\R^{4d})$, define
\begin{equation}\label{Coupling operator}
\begin{split}
&\big(\widetilde{\scr{L}}_{[\hat{\mu},\tilde{\mu}],t}f\big)\big((x,y),(x',y')\big)\\
&=\big(\widetilde{\scr{L}}_{[\hat{\mu},\tilde{\mu}],t,0}f\big)\big((x,y),(x',y')\big)+\big(\widetilde{\scr{L}}_{x,x'}\big)f\big((x,\cdot),(x',\cdot)\big)(y,y')\I_{\{\De((x,y),(x',y'))\le D_\Ga\}} \\
&\quad+\big(\widetilde{\scr{L}}^\ast_{x,x'}\big)f\big((x,\cdot),(x',\cdot)\big)(y,y')\I_{\{\De((x,y),(x',y'))>D_\Ga\}},
\end{split}\end{equation}
where $\widetilde{\scr{L}}_{x,x'}$ and $\widetilde{\scr{L}}^*_{x,x'}$ are given by \eqref{Lxx'} and \eqref{Lxx''} respectively, and
\begin{equation}\label{Ll}\begin{split}
&\big(\widetilde{\scr{L}}_{[\hat{\mu},\tilde{\mu}],t,0}f\big)\big((x,y),(x',y')\big)\\
&:=\bigg\langle b(x)+\int_{\R^d}\tilde{b}(x,z)\,\hat{\mu}_t^X(dz)-\ga y, \nabla_yf\big((x,y),(x',y')\big)\bigg\rangle+\big\langle y, \nabla_xf\big((x,y),(x',y')\big)\big\rangle\\
&\quad\,\,\,+\bigg\langle b(x')+\int_{\R^d}\tilde{b}(x',z')\,\tilde{\mu}^X_t(dz')-\ga y', \nabla_{y'}f\big((x,y),(x',y')\big)\bigg\rangle+\big\langle y', \nabla_{x'}f\big((x,y),(x',y')\big)\big\rangle.
\end{split}\end{equation}

\begin{proposition}\label{b}
The operator $\widetilde{\scr{L}}_{[\hat{\mu},\tilde{\mu}],t}$ defined by \eqref{Coupling operator} is a coupling operator of $\scr{L}_{[\hat{\mu}],t}$ and $\scr{L}_{[\tilde{\mu}],t}$ given by \eqref{L}; i.e.,  for any $h,\,g\in C^2_b(\R^{2d})$,
$$
\big(\widetilde{\scr{L}}_{[\hat{\mu},\tilde{\mu}],t}f\big) ((x,y),(x',y'))=(\scr{L}_{[\hat{\mu}],t}h)(x,y)+(\scr{L}_{[\tilde{\mu}],t}g)(x',y'),
$$
where $f ((x,y),(x',y') ):=h(x,y)+g(x',y')$.
\end{proposition}

\begin{proof}  Recall that $\scr{L}_{[{\mu}],t, 0}$ and $\scr{L}_1$ are defined
 in \eqref{L}. Then, it is easy to see that $\widetilde{\scr{L}}_{[\hat{\mu},\tilde{\mu}],t,0}$ defined by \eqref{Ll} is a coupling operator of
$\scr{L}_{[\hat{\mu}],t, 0}$ and $\scr{L}_{[\tilde{\mu}],t, 0}$, and $\widetilde{\scr{L}}^*_{x,x'}$ defined by \eqref{Lxx''} is a coupling operator of $\scr{L}_1$ and itself.
Therefore, it suffices to verify that $\widetilde{\scr{L}}_{x,x'}$ defined by \eqref{Lxx'} is a coupling operator of $\scr{L}_1$ and itself.

Firstly, let $f(y,y')=h(y)$ for any $y,\,y'\in\R^d$. It obviously holds that
\begin{align*}
\big(\widetilde{\scr{L}}_{x,x'}f\big)(y,y')
&=\frac{1}{2}\int_{\R^d}\big(h(y+z)-h(y)-\big\langle\nabla_yh(y),z\big\rangle\I_{\{|z|\le1\}}\big)\nu_{-\ga(q)_\ka}(dz)\\
&\quad+\frac{1}{2}\int_{\R^d}\big(h(y+z)-h(y)-\big\langle\nabla_yh(y),z\big\rangle\I_{\{|z|\le1\}}\big)\nu_{\ga(q)_\ka}(dz)\\
&\quad+\int_{\R^d}\big(h(y+z)-h(y)-\big\langle\nabla_yh(y),z\big\rangle\I_{\{|z|\le1\}}\big)\left(\nu-\frac{1}{2}\nu_{-\ga(q)_\ka}-\frac{1}{2}\nu_{\ga(q)_\ka}\right)(dz)\\
&=\scr{L}_1h(y).
\end{align*}
Secondly, let $f(y,y')=g(y')$ for any $y,\,y'\in\R^d$. Then
\begin{align*}
\big(\widetilde{\scr{L}}_{x,x'}f\big)(y,y')
&=\frac{1}{2}\int_{\R^d}\big(g(y'+z+\ga(q)_\ka)-g(y')-\big\langle\nabla_{y'}g(y'),z+\ga(q)_\ka\big\rangle\I_{\{|z+\ga(q)_\ka|\le1\}}\big)\nu_{-\ga(q)_\ka}(dz)\\
&\quad+\frac{1}{2}\int_{\R^d}\big(g(y'+z-\ga(q)_\ka)-g(y')-\big\langle\nabla_{y'}g(y'),z-\ga(q)_\ka\big\rangle\I_{\{|z-\ga(q)_\ka|\le1\}}\big)\nu_{\ga(q)_\ka}(dz)\\
&\quad+\int_{\R^d}\big(g(y'+z)-g(y')-\big\langle\nabla_{y'}g(y'),z\big\rangle\I_{\{|z|\le1\}}\big)\left(\nu-\frac{1}{2}\nu_{-\ga(q)_\ka}-\frac{1}{2}\nu_{\ga(q)_\ka}\right)(dz)\\
&=\frac{1}{2}\int_{\R^d}\big(g(y'+z)-g(y')-\big\langle\nabla_{y'}g(y'),z\big\rangle\I_{\{|z|\le1\}}\big)\nu_{\ga(q)_\ka}(dz)\\
&\quad+\frac{1}{2}\int_{\R^d}\big(g(y'+z)-g(y')-\big\langle\nabla_{y'}g(y'),z\big\rangle\I_{\{|z|\le1\}}\big)\nu_{-\ga(q)_\ka}(dz)\\
&\quad+\int_{\R^d}\big(g(y'+z)-g(y')-\big\langle\nabla_{y'}g(y'),z\big\rangle\I_{\{|z|\le1\}}\big)\left(\nu-\frac{1}{2}\nu_{-\ga(q)_\ka}-\frac{1}{2}\nu_{\ga(q)_\ka}\right)(dz)\\
&=\scr{L}_1g(y'),
\end{align*}
where in the second equality we have changed the variables $z+\ga(q)_\ka\mapsto z$ and $z-\ga(q)_\ka\mapsto z$, respectively, and used the facts
\begin{equation}\label{e}
\nu_{-\ga(q)_\ka}(d(z-\ga(q)_\ka))=\nu_{\ga(q)_\ka}(dz), \quad \nu_{\ga(q)_\ka}(d(z+\ga(q)_\ka))=\nu_{-\ga(q)_\ka}(dz).
\end{equation}
Combining with both equalities above, we get the desired assertion.
\end{proof}

\subsection{Coupling process for the decoupled SDE \eqref{decoupled1} }
In this subsection, we shall construct explicitly the coupling process associated with the coupling operator $\widetilde{\scr{L}}_{[\hat{\mu},\tilde{\mu}],t}$ defined by \eqref{Coupling operator}.

Firstly, by the L\'evy-It\^o decomposition, $(L_t)_{t\ge0}$ can be represented as below
$$
L_t=\int_0^t\int_{\{|z|>1\}}z\,N(ds,dz)+\int_0^t\int_{\{|z|\le 1\}}z\,\wt{N}(ds,dz), \quad t\ge0,
$$
where $N(ds,dz)$ is the Poisson random measure with the intensity measure $ds\nu(dz)$, and $\wt{N}(ds,dz)$ is its compensated Poisson random measure, i.e.,
$$
\wt{N}(ds,dz)=N(ds,dz)-ds\nu(dz).
$$
Following the ideas from \cite[Section 2.2]{LMW} and \cite[Section 2]{LW}, we extend $N$ from $\RR_+\times\RR^d$ to $\RR_+\times\RR^d\times[0,1]$ in the following way
$$
N(ds,dz,du)=\sum_{0\le s'\le s,\De L_{s'}\neq0}\de_{(s',\De L_{s'})}(ds,dz)\I_{[0,1]}(du).
$$
In particular, $(L_t)_{t\ge0}$ can be reformulated as
\begin{align*}
L_{t}&=\int_0^t\int_{\{|z|>1\}\times[0,1]}z\,N(ds,dz,du)+\int_0^t\int_{\{|z|\le 1\}\times[0,1]}z\,\wt{N}(ds,dz,du)\\
&=:\int_0^t\int_{\RR^d\times[0,1]}z\,\bar{N}(ds,dz,du), \quad t\ge0.
\end{align*}

Now we consider the following SDE
\begin{equation}\left\{\begin{array}{l}\label{Coupling process}
dX^{\hat\mu}_t=Y^{\hat\mu}_tdt,\\
dY^{\hat\mu}_t=\Big(b(X^{\hat\mu}_t)+\displaystyle\int_{\R^d}\tilde{b}(X^{\hat\mu}_t,z)\,\hat\mu_t^{X}(dz)-\ga Y^{\hat\mu}_t\Big)\,dt+dL_t,\\
dX^{\tilde\mu}_t=Y^{\tilde\mu}_tdt,\\
dY^{\tilde\mu}_t=\Big(b(X^{\tilde\mu}_t)+\displaystyle\int_{\R^d}\tilde{b}(X^{\tilde\mu}_t,z)\,\tilde\mu_t^{X}(dz)-\ga Y^{\tilde\mu}_t\Big)\,dt+dL^\ast_t.\end{array}\right.\\
\end{equation}
Herein,
\begin{align*}
L^*_{t}&:=\int_0^t\int_{\R^d\times[0,1]}z\I_{\{\De((X_{s-}^{\hat\mu},Y_{s-}^{\hat\mu}),(X_{s-}^{\tilde\mu},Y_{s-}^{\tilde\mu}))>D_\Ga\}}\bar{N}(ds,dz,du)\\
&\quad\,\,+\int_0^t\int_{\R^d\times[0,1]}\big[\big(z+\ga(Q_{s-})_{\ka}\big)\I_{\{u\le\frac{1}{2}\rho(-\ga (Q_{s-})_{\ka},z)\}}\\
&\qquad\qquad\qquad\quad~~+\big(z-\ga(Q_{s-})_{\ka}\big)\I_{\{\frac{1}{2}\rho(-\ga( Q_{s-})_{\ka},z)<u\le\frac{1}{2}[\rho(-\ga(Q_{s-})_{\ka},z)+\rho(\ga( Q_{s-})_{\ka},z)]\}}\\
&\qquad\qquad\qquad\quad~~+z\I_{\{\frac{1}{2}[\rho(-\ga(Q_{s-})_{\ka},z)+\rho(\ga( Q_{s-})_{\ka},z)]<u\le1\}}\big]\I_{\{\De((X_{s-}^{\hat\mu},Y_{s-}^{\hat\mu}),(X_{s-}^{\tilde\mu},Y_{s-}^{\tilde\mu}))\le D_\Ga\}}\bar{N}(ds,dz,du)\\
&\quad\,\,-\int_0^t\int_{\R^d\times[0,1]}\big[\big(z+\ga(Q_{s-})_{\ka}\big)\big(\I_{\{|z+\ga(Q_{s-})_{\ka}|\le1\}}-\I_{\{|z|\le1\}}\big)\I_{\{u\le\frac{1}{2}\rho(-\ga (Q_{s-})_{\ka},z)\}}\\
&\qquad\qquad+\big(z-\ga(Q_{s-})_{\ka}\big)\big(\I_{\{|z-\ga(Q_{s-})_{\ka}|\le1\}}-\I_{\{|z|\le1\}}\big)\I_{\{\frac{1}{2}\rho(-\ga( Q_{s-})_{\ka},z)<u\le\frac{1}{2}[\rho(-\ga(Q_{s-})_{\ka},z)+\rho(\ga( Q_{s-})_{\ka},z)]\}}\big]\\
&\qquad\qquad\qquad\qquad\quad\times \I_{\{\De((X_{s-}^{\hat\mu},Y_{s-}^{\hat\mu}),(X_{s-}^{\tilde\mu},Y_{s-}^{\tilde\mu}))\le D_\Ga\}}\,\nu(dz)\,ds\,du,
\end{align*}
where $$Q_{t}:=V_{t}+\ga^{-1}W_t \quad \hbox{with} \quad V_t:=X^{\hat\mu}_t-X^{\tilde \mu}_t \quad \hbox{and} \quad W_t:=Y^{\hat\mu}_t-Y^{\tilde\mu}_t,$$
and
\begin{equation}\label{gg}
\rho(x,z):=\frac{\nu_x(dz)}{\nu(dz)}, \quad x,z\in\R^d.
\end{equation}
A straightforward calculation shows
\begin{equation}\label{dd}
\begin{split}dL^\ast_t=&\I_{\{\De((X_{t-}^{\hat\mu},Y_{t-}^{\hat\mu}),(X_{t-}^{\tilde\mu},Y_{t-}^{\tilde\mu}))>D_\Ga\}}\,dL_t\\
&+\I_{\{\De((X_{t-}^{\hat\mu},Y_{t-}^{\hat\mu}),(X_{t-}^{\tilde\mu},Y_{t-}^{\tilde\mu}))\le D_\Ga\}}\bigg(dL_t+\ga(Q_{t-})_\ka\int_{\R^d\times[0,1]}\Xi(\ga(Q_{t-})_\ka,z,u)\,N(dt,dz,du)\bigg)\\
=& dL_t+\I_{\{\De((X_{t-}^{\hat\mu},Y_{t-}^{\hat\mu}),(X_{t-}^{\tilde\mu},Y_{t-}^{\tilde\mu}))\le D_\Ga\}} \ga(Q_{t-})_\ka\int_{\R^d\times[0,1]}\Xi(\ga(Q_{t-})_\ka,z,u)\,N(dt,dz,du),
\end{split}\end{equation}
where, for $x,z\in\R^d$ and $u\in[0,1]$,
\begin{equation}\label{ee}
\Xi(x,z,u):=\I_{\{u\le\frac{1}{2}\rho(-x,z)\}}-\I_{\{\frac{1}{2}\rho(-x,z)<u\le\frac{1}{2}(\rho(-x,z)+\rho(x,z))\}}.
\end{equation}
Hence, according to \eqref{dd}, \eqref{Coupling process} can be rewritten as
\begin{equation}\label{ff}
\begin{split}
\left\{\begin{array}{l}
dX^{\hat\mu}_t=Y^{\hat\mu}_t\,dt,\\
dY^{\hat\mu}_t=\Big(b(X^{\hat\mu}_t)+\displaystyle\int_{\R^d}\tilde{b}(X^{\hat\mu}_t,z)\,\hat\mu_t^{X}(dz)-\ga Y^{\hat\mu}_t\Big)\,dt+dL_t,\\
dX^{\tilde\mu}_t=Y^{\tilde\mu}_tdt,\\
dY^{\tilde\mu}_t=\Big(b(X^{\tilde\mu}_t)+\displaystyle\int_{\R^d}\tilde{b}(X^{\tilde\mu}_t,z)\,\tilde\mu_t^{X}(dz)-\ga Y^{\tilde\mu}_t\Big)\,dt+
dL_t\\
\qquad\quad\,\,\,+\I_{\{\De((X_{t-}^{\hat\mu},Y_{t-}^{\hat\mu}),(X_{t-}^{\tilde\mu},Y_{t-}^{\tilde\mu}))\le D_\Ga\}}
\ga(Q_{t-})_\ka\displaystyle\int_{\R^d\times[0,1]}\Xi(\ga(Q_{t-})_\ka,z,u)\,N(ds,dz,du)
.\end{array}\right.\\
\end{split}
\end{equation}

Recall that, under Assumptions \textbf{(A1)} and \textbf{(A2)} (as well as the condition that the L\'evy measure $\nu$ of the L\'evy process $(L_t)_{t\ge0}$ satisfies $\displaystyle\int_{\R^d}(|z|\wedge|z|^2)\,\nu(dz)<\infty$), the SDE \eqref{decoupled1} has a unique strong solution $(X^{\mu}_t,Y^{\mu}_t)_{t\ge0}$. Then, by following the proof of \cite[Proposition 2.2]{LW}, one can see that \eqref{ff} also has a unique strong solution $((X^{\hat\mu}_t,Y^{\hat\mu}_t),(X^{\tilde\mu}_t,Y^{\tilde\mu}_t))_{t\ge0}$. Furthermore, the following statement indicates that $((X^{\hat\mu}_t,Y^{\hat\mu}_t),(X^{\tilde\mu}_t,Y^{\tilde\mu}_t))_{t\ge0}$ is indeed a coupling process of $(X^{\mu}_t,Y^{\mu}_t)_{t\ge0}$.

\begin{proposition}\label{coupling proof}
The infinitesimal generator of the process $((X^{\hat\mu}_t,Y^{\hat\mu}_t),(X^{\tilde\mu}_t,Y^{\tilde\mu}_t))_{t\ge0}$ is just the coupling operator $\widetilde{\scr{L}}_{[\hat\mu,\tilde\mu],t}$ given by \eqref{Coupling operator}.~Consequently, $((X^{\hat\mu}_t,Y^{\hat\mu}_t),(X^{\tilde\mu}_t,Y^{\tilde\mu}_t))_{t\ge0}$ is a coupling process of $(X^{\mu}_t,Y^{\mu}_t)_{t\ge0}$.
\end{proposition}

\begin{proof}
Let $\overline{\scr L}$ be the infinitesimal generator of $((X^{\hat\mu}_t,Y^{\hat\mu}_t),(X^{\tilde\mu}_t,Y^{\tilde\mu}_t))_{t\ge0}$. Recall that $q=x-x'+\ga^{-1}(y-y')$ for any $x,\,x',\,y,\,y'\in\R^d$. Then, by the It\^{o}'s formula and \eqref{ff}, for any $f\in C_b^2(\R^{4d})$,
\begin{align*}
&(\overline{\scr L}f)((x,y),(x',y'))\\
&=(\widetilde{\scr{L}}_{[\hat{\mu},\tilde{\mu}],t,0}f)((x,y),(x',y'))\\
&\quad+\int_{\R^d\times[0,1]}\bigg(f\big((x,y+z),(x',y'+(z+\ga(q)_\ka\I_{\{\De((x,y),(x',y'))\le D_\Ga\}})\I_{\{u\le\frac{1}{2}\rho(-\ga (q)_{\ka},z)\}}\\
&\qquad\qquad\qquad\qquad+\big(z-\ga(q)_{\ka}\I_{\{\De((x,y),(x',y'))\le D_\Ga\}}\big)\I_{\{\frac{1}{2}\rho(-\ga(q)_{\ka},z)<u\le\frac{1}{2}[\rho(-\ga(q)_{\ka},z)+\rho(\ga(q)_{\ka},z)]\}}\\
&\qquad\qquad\qquad\qquad+z\I_{\{\frac{1}{2}[\rho(-\ga(q)_{\ka},z)+\rho(\ga(q)_{\ka},z)]<u\le1\}})\big)\\
&\qquad\qquad\qquad \quad-f((x,y),(x',y'))-\langle\nabla_yf((x,y),(x',y')),z\rangle\I_{\{|z|\le1\}}\\
&\qquad\qquad\qquad\qquad-\langle\nabla_{y'}f((x,y),(x',y')),(z+\ga(q)_\ka\I_{\{\De((x,y),(x',y'))\le D_\Ga\}})\I_{\{|z|\le1,u\le\frac{1}{2}\rho(-\ga (q)_{\ka},z)\}}\\
&\qquad\quad\qquad\qquad\qquad+(z-\ga(q)_{\ka}\I_{\{\De((x,y),(x',y'))\le D_\Ga\}})\I_{\{|z|\le1,\frac{1}{2}\rho(-\ga(q)_{\ka},z)<u\le\frac{1}{2}[\rho(-\ga(q)_{\ka},z)+\rho(\ga(q)_{\ka},z)]\}}\\
&\qquad\qquad\qquad\qquad\quad+z\I_{\{|z|\le1,\frac{1}{2}[\rho(-\ga(q)_{\ka},z)+\rho(\ga(q)_{\ka},z)]<u\le1\}}\rangle\bigg)\nu(dz)du\\
&\quad-\int_{\R^d\times[0,1]}\big\langle\nabla_{y'}f((x,y),(x',y')),\big((z+\ga(q)_\ka\I_{\{\De((x,y),(x',y'))\le D_\Ga\}})\big(\I_{\{|(z+\ga(q)_\ka)|\le1\}}-\I_{\{|z|\le1\}}\big)\I_{\{u\le\frac{1}{2}\rho(-\ga (q)_{\ka},z)\}}\\
&\qquad\qquad\qquad\qquad\qquad\qquad\quad+(z+\ga(q)_\ka\I_{\{\De((x,y),(x',y'))\le D_\Ga\}})\big(\I_{\{|(z+\ga(q)_\ka)|\le1\}}-\I_{\{|z|\le1\}}\big)\\
&\qquad\qquad\qquad\qquad\qquad\qquad\qquad\times \I_{\{\frac{1}{2}\rho(-\ga(q)_{\ka},z)<u\le\frac{1}{2}[\rho(-\ga(q)_{\ka},z)+\rho(\ga(q)_{\ka},z)]\}}\big)\big\rangle\,\nu(dz)\,du\\
&=(\widetilde{\scr{L}}_{[\hat{\mu},\tilde{\mu}],t,0}f)((x,y),(x',y'))\\
&\quad +\I_{\{\De((x,y),(x',y'))\le D_\Ga\}} \int_{\R^d\times[0,1]}\bigg(f\big((x,y+z),(x',y'+(z+\ga(q)_\ka)\I_{\{u\le\frac{1}{2}\rho(-\ga (q)_{\ka},z)\}}\\
&\quad\qquad\qquad\qquad\qquad\qquad\qquad\qquad\quad\,\,+\big(z-\ga(q)_{\ka}\big)\I_{\{\frac{1}{2}\rho(-\ga(q)_{\ka},z)<u\le\frac{1}{2}[\rho(-\ga(q)_{\ka},z)+\rho(\ga(q)_{\ka},z)]\}}\\
&\quad\qquad\qquad\qquad\qquad\qquad\qquad\qquad\quad\,\,+z\I_{\{\frac{1}{2}[\rho(-\ga(q)_{\ka},z)+\rho(\ga(q)_{\ka},z)]<u\le1\}})\big)\\
&\quad\qquad\qquad\qquad\qquad\qquad\qquad\qquad\,\,-f((x,y),(x',y'))-\langle\nabla_yf((x,y),(x',y')),z\rangle\I_{\{|z|\le1\}}\\
&\qquad\qquad\qquad\qquad\qquad\qquad\qquad\quad\,\,-\langle\nabla_{y'}f((x,y),(x',y')),(z+\ga(q)_\ka)\I_{\{|z|\le1,u\le\frac{1}{2}\rho(-\ga (q)_{\ka},z)\}}\\
&\qquad\qquad\qquad\qquad\qquad\qquad\qquad\qquad\,\,+(z-\ga(q)_{\ka})\I_{\{|z|\le1,\frac{1}{2}\rho(-\ga(q)_{\ka},z)<u\le\frac{1}{2}[\rho(-\ga(q)_{\ka},z)+\rho(\ga(q)_{\ka},z)]\}}\\
&\qquad\qquad\qquad\qquad\qquad\qquad\qquad\qquad\,\,+z\I_{\{|z|\le1,\frac{1}{2}[\rho(-\ga(q)_{\ka},z)+\rho(\ga(q)_{\ka},z)]<u\le1\}}\rangle\bigg)\nu(dz)du\\
&\quad-\I_{\{\De((x,y),(x',y'))\le D_\Ga\}}\int_{\R\times[0,1]}\big\langle\nabla_{y'}f((x,y),(x',y')),\big((z+\ga(q)_\ka)\big(\I_{\{|(z+\ga(q)_\ka)|\le1\}}-\I_{\{|z|\le1\}}\big)\I_{\{u\le\frac{1}{2}\rho(-\ga (q)_{\ka},z)\}}\\
&\qquad\qquad\qquad\qquad\qquad\qquad\qquad\qquad\qquad\qquad\,\, +(z+\ga(q)_\ka)\big(\I_{\{|(z+\ga(q)_\ka)|\le1\}}-\I_{\{|z|\le1\}}\big)\\
&\qquad\qquad\qquad\qquad\qquad\qquad\qquad\qquad\qquad\qquad\qquad \times \I_{\{\frac{1}{2}\rho(-\ga(q)_{\ka},z)<u\le\frac{1}{2}[\rho(-\ga(q)_{\ka},z)+\rho(\ga(q)_{\ka},z)]\}}\big)\big\rangle\,\nu(dz)\,du\\
&\quad+\I_{\{\De((x,y),(x',y'))> D_\Ga\}}\int_{\R^d\times[0,1]}\bigg(f\big((x,y+z),(x',y'+z)\big)-f((x,y),(x',y'))\\
&\qquad\qquad\qquad\qquad\qquad\qquad-\langle\nabla_yf((x,y),(x',y')),z\rangle\I_{\{|z|\le1\}}-\langle\nabla_{y'}f((x,y),(x',y')),z\rangle\I_{\{|z|\le1\}}\bigg)\nu(dz)du\\
&=(\widetilde{\scr{L}}_{[\hat{\mu},\tilde{\mu}],t,0}f)((x,y),(x',y'))\\
&\quad +\I_{\{\De((x,y),(x',y'))\le D_\Ga\}} \int_{\R^d\times[0,1]}\bigg(f\big((x,y+z),(x',y'+(z+\ga(q)_\ka)\I_{\{u\le\frac{1}{2}\rho(-\ga (q)_{\ka},z)\}}\\
&\quad\qquad\qquad\qquad\qquad\qquad\qquad\qquad\quad\,\,+\big(z-\ga(q)_{\ka}\big)\I_{\{\frac{1}{2}\rho(-\ga(q)_{\ka},z)<u\le\frac{1}{2}[\rho(-\ga(q)_{\ka},z)+\rho(\ga(q)_{\ka},z)]\}}\\
&\quad\qquad\qquad\qquad\qquad\qquad\qquad\qquad\quad\,\,+z\I_{\{\frac{1}{2}[\rho(-\ga(q)_{\ka},z)+\rho(\ga(q)_{\ka},z)]<u\le1\}})\big)\\
&\quad\qquad\qquad\qquad\qquad\qquad\qquad\qquad\,\,-f((x,y),(x',y'))-\langle\nabla_yf((x,y),(x',y')),z\rangle\I_{\{|z|\le1\}}\\
&\qquad\qquad\qquad\qquad\qquad\qquad\qquad\quad\,\,-\langle\nabla_{y'}f((x,y),(x',y')),(z+\ga(q)_\ka)\I_{\{|z+\ga(q)_{\ka}|\le1,u\le\frac{1}{2}\rho(-\ga (q)_{\ka},z)\}}\\
&\qquad\qquad\qquad\qquad\qquad\qquad\qquad\qquad\,\,+(z-\ga(q)_{\ka})\I_{\{|z-\ga(q)_{\ka}|\le1,\frac{1}{2}\rho(-\ga(q)_{\ka},z)<u\le\frac{1}{2}[\rho(-\ga(q)_{\ka},z)+\rho(\ga(q)_{\ka},z)]\}}\\
&\qquad\qquad\qquad\qquad\qquad\qquad\qquad\qquad\,\,+z\I_{\{|z|\le1,\frac{1}{2}[\rho(-\ga(q)_{\ka},z)+\rho(\ga(q)_{\ka},z)]<u\le1\}}\rangle\bigg)\nu(dz)du\\
&\quad+\I_{\{\De((x,y),(x',y'))> D_\Ga\}}\int_{\R^d\times[0,1]}\bigg(f\big((x,y+z),(x',y'+z)\big)-f((x,y),(x',y'))\\
&\qquad\qquad\qquad\qquad\qquad\qquad-\langle\nabla_yf((x,y),(x',y')),z\rangle\I_{\{|z|\le1\}}-\langle\nabla_{y'}f((x,y),(x',y')),z\rangle\I_{\{|z|\le1\}}\bigg)\nu(dz)du\\
&=\big(\widetilde{\scr{L}}_{[\hat{\mu},\tilde{\mu}],t,0}f\big)\big((x,y),(x',y')\big)+\big(\widetilde{\scr{L}}_{x,x'}\big)f\big((x,\cdot),(x',\cdot)\big)(y,y')\I_{\{\De((x,y),(x',y'))\le D_\Ga\}}\\
&\quad+\big(\widetilde{\scr{L}}^\ast_{x,x'}\big)f\big((x,\cdot),(x',\cdot)\big)(y,y')\I_{\{\De((x,y),(x',y'))>D_\Ga\}}\\
&=\big(\widetilde{\scr{L}}_{[\hat\mu,\tilde\mu],t}f\big)\big((x,y),(x',y')\big),
\end{align*} where in the
fourth equality we used
\eqref{gg}.
The proof is completed.
\end{proof}

\section{Proof of Theorem \ref{Main theorem}} \label{Proofs}
In this part, set
\begin{equation}\label{A}
\ps(r)=
\begin{cases}
\int_0^{r}e^{-g(s)}\,ds, &\quad r\in [0,2R_1], \\
\ps(2R_1)+\ps'(2R_1)(r-2R_1), & \quad r\in (2R_1,\infty),
\end{cases}
\end{equation}
where \begin{equation}\label{g}
g(r):=2\al\ga(1+k_0\al)\int_0^r\si\Big(\frac{s}{1+k_0\al}\Big)^{-1}\,ds,\quad r\in(0,2R_1],
\end{equation}
$k_0>4$, $\alpha=2L_b\gamma^{-2}$, $\sigma(r)$ is given in  Assumption \textbf{(A0)}, and $R_1$ is chosen later.
Thanks to Assumption \textbf{(A0)}, $$g'(r)\ge0,\,\,~g''(r)\le0,\,\,g'''(r)\ge0,\quad r\in(0,2R_1],$$ and so
for any $0\le\delta\le r\le R_1$,
\begin{equation}\label{e:function2}
\ps(r+\delta)+\ps(r-\delta)-2\ps(r)\le \ps''(r)\delta^2.
\end{equation} On the other hand, one can see that $\ps$ is an increasing concave function on $[0,\infty)$. Thus, for any $0\le\delta\le r$,
\begin{equation}\label{e:function1}
\ps(r+\delta)+\ps(r-\delta)-2\ps(r)\le0,
\end{equation}
and for all $r\ge0$
\begin{equation}\label{inequality}
\ps'(2R_1)r\le\ps(r)\le
r.
\end{equation} See \cite[Lemma 4.1]{LW} for the details.

Recall that for any $x$, $x'$, $y$, $y'\in\R^d$,
$$
\De((x,y),(x',y')):=r_s((x,y),(x',y'))-\ep r_l((x,y),(x',y')),
$$ where $$ r_s((x,y),(x',y')):=\al|v|+|q|,  \quad r^2_l((x,y),(x',y')):=A|v|^2+B\langle v,w\rangle+C|w|^2
$$ are defined by \eqref{e:metric1} and with $$
v:=x-x', \quad w:=y-y', \quad q:=v+\ga^{-1}w,$$  and the constant $\epsilon\in (0,1/2]$ so that $2\ep r_l((x,y),(x',y'))\le r_s((x,y),(x',y'))$ for any $x$, $x'$, $y$, $y'\in\R^d$.
In particular, it holds that
 $$ r_s((x,y),(x',y'))/2\le \De((x,y),(x',y'))\le r_s((x,y),(x',y')).$$
  We then fix $$R_1:=\sup_{((x,y),(x',y')):\De((x,y),(x',y'))\le D_\Ga}r_s((x,y),(x',y')),$$ where $D_\Ga$ is given by \eqref{e:Gamma} and with $\cR$ therein fixed later.

Next, define the metric $\rho:\R^{2d}\times\R^{2d}\rightarrow[0,\infty)$ by
\begin{equation}\label{Metric}
\rho((x,y),(x',y')):=\ps((\De((x,y),(x',y'))\wedge D_\Ga)+\ep r_l((x,y),(x',y'))).
\end{equation}
In particular, the metric $\rho$ achieves a continuous transition between $r_s((x,y),(x',y'))$ and $r_l((x,y),(x',y'))$ by considering $r_s((x,y),(x',y'))\wedge (D_\Ga+\ep r_l((x,y),(x',y')))$. More precisely,
$$
\rho((x,y),(x',y'))=
\begin{cases}
\ps(r_s((x,y),(x',y'))), &\quad \De((x,y),(x',y'))\le D_\Ga,\\
\ps(D_\Ga+\ep r_l((x,y),(x',y'))), & \quad \De((x,y),(x',y'))>D_\Ga.
\end{cases}
$$
Observe that $r_s((x,y),(x',y'))\le R_1$ when $\De((x,y),(x',y'))\le D_\Ga$, and $r^2_l((x,y),(x',y'))> \cR$ when $\De((x,y),(x',y'))> D_\Ga$.
In the following, for simplicity we write $r_s((x,y),(x',y'))$ and $r_l((x,y),(x',y'))$ as $r_s$ and $r_l$ without confusion.

With all the notations above, we start from the following simple lemma.
\begin{lemma}\label{le1}
For the coupling operator $\widetilde{\scr{L}}_{[\hat{\mu},\tilde{\mu}],t}$ given by \eqref{Coupling operator}, it holds that
\begin{equation}\label{h}\begin{split}
(\widetilde{\scr{L}}_{[\hat{\mu},\tilde{\mu}],t}\,\ps)(r_s)&=\ps'(r_s)\bigg\{\frac{1}{|q|}\bigg\langle q,\ga^{-1}(b(x)-b(x')+\int_{\R^d}\tilde{b}(x,z)\,\hat{\mu}^X_t(dz)-\int_{\R^d}\tilde{b}(x',z')\,\tilde{\mu}^X_t(dz'))\bigg\rangle\\
&\qquad\qquad\,\, -\al\ga|v|+\frac{\al\ga}{|v|}\langle v,q\rangle\bigg\}\\
&\quad+\frac{1}{2}\big(\ps(r_s-(\ka\wedge|q|))+\ps(r_s+(\ka\wedge|q|))-2\ps(r_s)\big)\nu_{\ga(q)_\ka}(\R^d)\I_{\{\Delta((x,y),(x',y'))\le D_\Gamma\}}\\
&=:\Theta_1+\Theta_2.
\end{split}\end{equation}
\end{lemma}

\begin{proof}
Note that
$$
\nabla_y\ps(r_s)=\ps'(r_s)\frac{q}{\ga|q|}, \quad \nabla_{y'}\ps(r_s)=-\ps'(r_s)\frac{q}{\ga|q|},
$$
$$
\nabla_x\ps(r_s)=\ps'(r_s)\Big(\al\frac{v}{|v|}+\frac{q}{|q|}\Big), \quad
\nabla_{x'}\ps(r_s)=-\ps'(r_s)\Big(\al\frac{v}{|v|}+\frac{q}{|q|}\Big).
$$
Recall that $\widetilde{\scr{L}}_{[\hat{\mu},\tilde{\mu}],t,0}$, $\widetilde{\scr{L}}_{x,x'}$ and $\widetilde{\scr{L}}^\ast_{x,x'}$ are given by \eqref{Ll}, \eqref{Lxx'} and \eqref{Lxx''}, respectively.

Firstly,
\begin{align*}
&(\widetilde{\scr{L}}_{[\hat{\mu},\tilde{\mu}],t,0}\ps)(r_s)\\
&=\ps'(r_s)\left\{\frac{1}{|q|}\bigg\langle q,\ga^{-1}(b(x)-b(x')+\int_{\R^d}\tilde{b}(x,z)\,\hat{\mu}_t^X(dz)-\int_{\R^d}\tilde{b}(x',z')\,\tilde{\mu}_t^X(dz'))\bigg\rangle+\frac{\al}{|v|}\langle v,w\rangle\right\}\\
&=\ps'(r_s)\left\{\frac{1}{|q|}\bigg\langle q,\ga^{-1}(b(x)-b(x')+\int_{\R^d}\tilde{b}(x,z)\,\hat{\mu}_t^X(dz)-\int_{\R^d}\tilde{b}(x',z')\,\tilde{\mu}_t^X(dz'))\bigg\rangle-\al\ga|v|+\frac{\al\ga}{|v|}\langle v,q\rangle\right\},
\end{align*}
where in the second identity we used $q=v+\ga^{-1}w$.

Secondly,
\begin{align*}
(\widetilde{\scr{L}}_{x,x'}\,\ps)(r_s)
&=\frac{1}{2}\int_{\R^d}\big(\ps(\al|v|+|q-(q)_\ka|)-\ps(r_s)-\frac{\ps'(r_s)}{\ga|q|}\langle q,z\rangle\I_{\{|z|\le1\}}\\
&\qquad\qquad\,+\frac{\ps'(r_s)}{\ga|q|}\langle q,z+\ga(q)_\ka\rangle\I_{\{|z+\ga(q)_\ka|\le1\}}\big)\,\nu_{-\ga(q)_\ka}(dz)\\
&\quad+\frac{1}{2}\int_{\R^d}\big(\ps(\al|v|+|q+(q)_\ka|)-\ps(r_s)-\frac{\ps'(r_s)}{\ga|q|}\langle q,z\rangle\I_{\{|z|\le1\}}\\
&\qquad\qquad\quad+\frac{\ps'(r_s)}{\ga|q|}\langle q,z-\ga(q)_\ka\rangle\I_{\{|z-\ga(q)_\ka|\le1\}}\big)\,\nu_{\ga(q)_\ka}(dz)\\
&=\frac{1}{2}\big(\ps(\al|v|+|q-(q)_\ka|)+\ps(\al|v|+|q+(q)_\ka|)-2\ps(r_s)\big)\nu_{\ga(q)_\ka}(\R^d)\\
&=\frac{1}{2}\big(\ps(r_s-(\ka\wedge|q|))+\ps(r_s+(\ka\wedge|q|))-2\ps(r_s)\big)\nu_{\ga(q)_\ka}(\R^d),
\end{align*}
where in the second equality we have changed the variables $z+\ga(q)_\ka\mapsto z$ and $z-\ga(q)_\ka\mapsto z$ respectively, and used  \eqref{e}.

Lastly, it is easy to see that
$$(\widetilde{\scr{L}}^\ast_{x,x'}\,\ps)(r_s)=0.$$

Therefore, the desired assertion follows from all the conclusions above.
\end{proof}

We obtain a local contraction result in the following statement. Recall that $k_0>4$ and \begin{equation}\label{al}
\al:=2L_b\ga^{-2}.
\end{equation}

\begin{proposition}\label{Proposition1}
Suppose that Assumptions {\rm\textbf{(A0)}}--{\rm\textbf{(A2)}} hold, then it holds that for any $x$, $x'$, $y$, $y'\in\R^d$ with $\De((x,y),(x',y'))\le D_\Ga$,
\begin{equation}\label{Lc}
\widetilde{\scr{L}}_{[\hat{\mu},\tilde{\mu}],t}\,\ps(r_s)\le-c_1\ps(r_s)-\frac{\al\ga}{4}\ps'(R_1)|v|+\ga^{-1}L_{\tilde{b}}(|v|+\bb{W}_1(\hat{\mu}_t^X,\tilde{\mu}_t^X)),
\end{equation}
where
\begin{equation}\label{c1}
c_1=\min\bigg\{1, \frac{k_0-4}{4(k_0\al+1)} \bigg\}\al\ga\ps'(R_1).
\end{equation}
\end{proposition}

\begin{proof}
Here, we will adopt the shorthand notations $\Theta_1$ and $\Theta_2$ introduced in \eqref{h}. Note that
\begin{align*}
\Theta_1
&\le \ps'(r_s)\left(\ga^{-1}|b(x)-b(x')|+\ga^{-1}\Big|\int_{\R^d}\tilde{b}(x,z)\,\hat{\mu}_t^X(dz)-\int_{\R^d}\tilde{b}(x',z')\,\tilde{\mu}_t^X(dz')\Big|-\al\ga|v|+\al\ga|q|\right)\\
&\le \ps'(r_s)\left(\ga^{-1}L_b|v|+\ga^{-1}L_{\tilde{b}} (|v|+\bb{W}_1(\hat{\mu}_t^X,\tilde{\mu}_t^X))-\al\ga|v|+\al\ga|q|\right)\\
&= \ps'(r_s)\left((L_b\ga^{-2}-\al)\ga|v|+\al\ga|q|+\ga^{-1}L_{\tilde{b}} (|v|+\bb{W}_1(\hat{\mu}_t^X,\tilde{\mu}_t^X))\right)\\
&= \ps'(r_s)\left(-\frac{\al\ga}{2}|v|+\al\ga|q|+\ga^{-1}L_{\tilde{b}}  (|v|+\bb{W}_1(\hat{\mu}_t^X,\tilde{\mu}_t^X))\right),
\end{align*}
where in the second inequality we used the Lipschitz continuity of $b$ in Assumption \textbf{(A1)} and \eqref{e:addremark}, and the last equality follows from \eqref{al}.

As mentioned above, since $\De((x,y),(x',y'))\le D_\Ga$, $r_s\le R_1$.
We then split the proof into two cases.

\noindent \textbf{ Case 1: $|v|\le k_0|q|$}. In this case, by $r_s:=\al|v|+|q|$, it holds that
$$|q|\le r_s\le(1+k_0\al)|q|.$$
Thus,
$$
\Theta_1\le \al\ga\ps'(r_s)r_s-\frac{\al\ga}{2}\ps'(r_s)|v|+\ps'(r_s)\ga^{-1}L_{\tilde{b}} (|v|+\bb{W}_1(\hat{\mu}_t^X,\tilde{\mu}_t^X)).
$$
Then, according to \eqref{e:function2} and \eqref{f} in Assumption \textbf{(A0)} (in particular the increasing property of the function $\sigma(r)$), we obtain
\begin{align*}
\widetilde{\scr{L}}_{[\hat{\mu},\tilde{\mu}],t}\,\ps(r_s)
&\le
\frac{1}{2}\ps''(r_s)J(\gamma(\ka\wedge|q|))(\ka\wedge|q|)^2+\al\ga\ps'(r_s)r_s-\frac{\al\ga}{2}\ps'(r_s)|v|+\ps'(r_s)\ga^{-1}L_{\tilde{b}}(|v|+\bb{W}_1(\hat{\mu}_t^X,\tilde{\mu}_t^X))\\
&\le
 \ps''(r_s)\sigma(|q|)|q|+\al\ga\ps'(r_s)r_s-\frac{\al\ga}{2}\ps'(r_s)|v|+\ps'(r_s)\ga^{-1}L_{\tilde{b}}(|v|+\bb{W}_1(\hat{\mu}_t^X,\tilde{\mu}_t^X))\\
&\le
\ps''(r_s)\si\Big(\frac{r_s}{1+k_0\al}\Big)\frac{1}{1+k_0\al}r_s+\al\ga\ps'(r_s)r_s-\frac{\al\ga}{2}\ps'(r_s)|v|+\ps'(r_s)\ga^{-1}L_{\tilde{b}}(|v|+\bb{W}_1(\hat{\mu}_t^X,\tilde{\mu}_t^X))\\
&\le
-\al\ga\ps'(r_s) r_s-\frac{\al\ga}{2}\ps'(r_s)|v|+\ps'(r_s)\ga^{-1}L_{\tilde{b}}(|v|+\bb{W}_1(\hat{\mu}_t^X,\tilde{\mu}_t^X)).
\end{align*}

\noindent \textbf{ Case 2: $|v|> k_0|q|$}. In this case, we have
$$\al|v|\le r_s\le \big(\al+\frac{1}{k_0}\big)|v|.$$
Thus, by \eqref{e:function1}, $\Theta_2\le0$, so
\begin{align*}
\widetilde{\scr{L}}_{[\hat{\mu},\tilde{\mu}],t}\,\ps(r_s)
&\le\ps'(r_s)\left\{-\frac{\al\ga}{4}|v|+\frac{\al\ga}{k_0}|v|-\frac{\al\ga}{4}|v|+\ga^{-1}L_{\tilde{b}} (|v|+\bb{W}_1(\hat{\mu}_t^X,\tilde{\mu}_t^X))\right\}\\
&\le-\frac{k_0-4}{4(k_0\al+1)}\al\ga\ps'(r_s)r_s-\frac{\al\ga}{4}\ps'(r_s)|v|+\ps'(r_s)\ga^{-1}L_{\tilde{b}}(|v|+\bb{W}_1(\hat{\mu}_t^X,\tilde{\mu}_t^X)).
\end{align*}

Combining with two cases above and applying the facts that $r\mapsto \ps'(r)$ is decreasing on $(0,R_1]$ and, for all $0<r\le R_1$, $\ps'(R_1)r\le \ps(r)\le r$, we find
\begin{align*}
\widetilde{\scr{L}}_{[\hat{\mu},\tilde{\mu}],t}\,\ps(r_s)
&\le-c_1\ps(r_s)-\frac{\al\ga}{4}\ps'(R_1)|v|+\ps'(r_s)\ga^{-1}L_{\tilde{b}}(|v|+\bb{W}_1(\hat{\mu}_t^X,\tilde{\mu}_t^X))\\
&\le-c_1\ps(r_s)-\frac{\al\ga}{4}\ps'(R_1)|v|+\ga^{-1}L_{\tilde{b}}(|v|+\bb{W}_1(\hat{\mu}_t^X,\tilde{\mu}_t^X)),
\end{align*}
where in the last inequality we also used $\ps'(r_s)\le1$ for all $r_s\in(0,R_1]$. The proof is completed.
\end{proof}

On the other hand, for any $x$, $x'$, $y$, $y'\in\R^d$ with $\De((x,y),(x',y'))>D_\Ga$ we obtain the following contraction result with the aid of the synchronous coupling. In particular, in the proof of the proposition below the coefficients $A$, $B$ and $C$ in the definition \eqref{e:metric1} of the metric $r_l$ are given explicitly, and the positive constant $\cR$ in \eqref{e:Gamma} is also fixed.

\begin{proposition}\label{Proposition2}
Suppose that Assumptions {\rm\textbf{(A1)}}--{\rm\textbf{(A2)}} and \eqref{hh} hold, then it holds that for any $x$, $x'$, $y$, $y'\in\R^d$ with $\De((x,y),(x',y'))>D_\Ga$,
\begin{equation}\label{ll}
\widetilde{\scr{L}}_{[\hat{\mu},\tilde{\mu}],t}\,r^2_l \le-\tau\ga r^2_l+\ga^{-1}|(1-2\tau)v+2\ga^{-1}w|L_{\tilde{b}}(|v|+\bb{W}_1(\hat{\mu}_t^X,\tilde{\mu}_t^X)),
\end{equation} where
\begin{equation}\label{tau}
\tau:=\min\{{1}/{8}, \ga^{-2}\th-4L^2_b\ga^{-4}/3\}.
\end{equation}
\end{proposition}

\begin{proof}
Recalling that $r^2_l:=A|v|^2+B\langle v,w\rangle+C|w|^2$, it holds that
$$
\nabla_yr^2_l=Bv+2Cw, \quad \nabla_{y'}r^2_l=-(Bv+2Cw),
$$
$$
\nabla_xr^2_l=2Av+Bw, \quad \nabla_{x'}r^2_l=-(2Av+Bw),
$$
so $\widetilde{\scr{L}}^\ast_{x,x'}\,r^2_l=0.$ By \eqref{Coupling operator},
\begin{align*}
\widetilde{\scr{L}}_{[\hat{\mu},\tilde{\mu}],t}\,r^2_l&=\widetilde{\scr{L}}_{[\hat{\mu},\tilde{\mu}],t,0}r^2_l\\
&=\bigg\langle Bv+2Cw,-\ga w+b(x)-b(x')+\int_{\R^d}\tilde{b}(x,z)\,\hat{\mu}_t^X(dz)-\int_{\R^d}\tilde{b}(x',z')\,\tilde{\mu}_t^X(dz')\bigg\rangle+\langle2Av+Bw,w\rangle\\
&=B\langle v,-\ga w\rangle+2A\langle v,w\rangle+2C\langle w,-\ga w\rangle+B\langle w,w\rangle+B\langle v,b(x)-b(x')\rangle+2C\langle w,b(x)-b(x')\rangle\\
&\quad+\bigg\langle Bv+2Cw,\int_{\R^d}\tilde{b}(x,z)\,\hat{\mu}_t^X(dz)-\int_{\R^d}\tilde{b}(x',z')\,\tilde{\mu}_t^X(dz')\bigg\rangle\\
&\le-\big[(\th B-L^2_b\ga^{-1}C)|v|^2+(\ga B-2A)\langle v,w\rangle+(\ga C-B)|w|^2\big]\\
&\quad+(BL_b+B\th )|v|^2\I_{\{|v|\le R_0\}}+|Bv+2Cw|L_{\tilde {b}}(|v|+\bb{W}_1(\hat{\mu}_t^X,\tilde{\mu}_t^X)).
\end{align*}
Here in the inequality above, we used the following facts that,  by Assumption \textbf{(A1)},
$$
\langle v,b(x)-b(x')\rangle\le-\th |v|^2\I_{\{|v|> R_0\}}+L_b|v|^2\I_{\{|v|\le R_0\}}=-\th |v|^2+(L_b+\th )|v|^2\I_{\{|v|\le R_0\}}
$$
and, by $2ab\le a^2+b^2$,
$$2\langle w,b(x)-b(x')\rangle\le2 L_b|w||v|\le L^2_b\ga^{-1} |v|^2+\ga  |w|^2.$$

Next, we aim to find positive constants $A$, $B$ and $C$ so that
\begin{equation}\label{la}
-\big[(\th B-L^2_b\ga^{-1}C)|v|^2+(\ga B-2A)\langle v,w\rangle+(\ga C-B)|w|^2\big]\le-\la\big[A|v|^2+B\langle v,w\rangle+C|w|^2\big]
\end{equation}
holds for some $\lambda>0$.
For this, it is sufficient to require that
\begin{equation}\left\{\begin{array}{l}\label{i}
\th B-L^2_b\ga^{-1}C\ge \la A\\
\ga B-2A=\la B\\
\ga C-B=\la C.\end{array}\right.\\
\end{equation}
Below, we take $\la:=2\tau\ga$ and $\tau:=\min\{{1}/{8}, \ga^{-2}\th-4L^2_b\ga^{-4}/3\}>0$. Set $A=\frac{1}{2}(1-2\tau)^2$, $B=(1-2\tau)\ga^{-1}$ and $C=\ga^{-2}$ in \eqref{i}; that is,
\begin{equation}\label{rl}
 r^2_l=\frac{1}{2}(1-2\tau)^2|v|^2+(1-2\tau)\ga^{-1}\langle v,w\rangle+\ga^{-2}|w|^2.
\end{equation}
Indeed,
with the choice of $A,B$ and $C$ above, it is easy to see that both equalities in \eqref{i} are satisfied. Furthermore,  by the definition of $\tau:=\min\{{1}/{8}, \ga^{-2}\th-4L^2_b\ga^{-4}/3\}>0$
(thanks to \eqref{hh}), $$\theta \gamma^{-2}\ge 4L_b^2\gamma^{-4}/3+\tau(1-2\tau)\ge  L_b^2\gamma^{-4}(1-2\tau)^{-1} +\tau(1-2\tau)$$ and so $\theta\gamma^{-2}(1-2\tau)\ge L_b^2\gamma^{-4}+\tau(1-2\tau)^2,$ which implies that the first inequality in \eqref{i} holds.
Thus,
\begin{equation}\label{cc}\begin{split}
\widetilde{\scr{L}}_{[\hat{\mu},\tilde{\mu}],t}\,r^2_l&\le
-2\tau\ga r^2_l+(L_b+\th)(1-2\tau)\ga^{-1}|v|^2\I_{\{|v|\le R_0\}}\\
&\quad
+\ga^{-1}|(1-2\tau)v+2\ga^{-1}w|L_{\tilde{b}}(|v|+\bb{W}_1(\hat{\mu}_t^X,\tilde{\mu}_t^X)).\end{split}
\end{equation}

Since $\De((x,y),(x',y'))> D_\Ga$, it holds that $r^2_l>\cR$. Choosing
\begin{equation}\label{l}
\cR:=\tau^{-1}\ga^{-2}(1-2\tau)(L_b+\th)R_0^2,
\end{equation}
we can see
$$
-\tau\ga r^2_l+(L_b+\th)(1-2\tau)\ga^{-1}|v|^2\I_{\{|v|\le R_0\}}\le-\tau\ga\cR+(L_b+\th)(1-2\tau)\ga^{-1}R_0^2=0.
$$
This along with \eqref{cc} yields the desired assertion.
\end{proof}

To prove Theorem \ref{Main theorem}, we will put two contraction results in Propositions \ref{Proposition1} and \ref{Proposition2} together. For simplicity, for any $t>0$, we write $r_s(t)=r_s((X^{\mu}_t,Y^{\mu}_t),(X^{\bar \mu}_t,Y^{\bar\mu}_t))$, $r_l(t)=r_l((X^{\mu}_t,Y^{\mu}_t),(X^{\bar\mu}_t,Y^{\bar\mu}_t))$, $\Delta(t)=\Delta((X^{\mu}_t,Y^{\mu}_t),(X^{\bar\mu}_t,Y^{\bar\mu}_t))$ and $\rho(t)=\psi((\De(t)\wedge D_\Ga)+\ep r_l(t))$.

\begin{proof}[Proof of Theorem $\ref{Main theorem}$]
Firstly, we shall verify \eqref{Wde}, and for this we distinguish two cases.\\
\noindent \textbf{ Case 1: $\De(t)\le D_\Ga$}. In this case, $r_s(t)\le R_1$ and $\rho(t)=\ps(r_s(t))$. By \eqref{Lc}, we obtain
\begin{equation}\label{v}\begin{split}
d\rho(t)&\le (\widetilde{\scr{L}}_{[\mu,\bar\mu],t}\,\ps)(r_s(t))\,dt+dM_t\\
&\le-c_1\ps(r_s(t))\,dt-\frac{\al\ga}{4}\ps'(R_1)|V_t|\,dt+\ga^{-1}L_{\tilde{b}}\big(|V_t|+\bb{E}|V_t|\big)\,dt+dM_t,\\
&\le-c_1\rho(t)\,dt-\frac{\al\ga}{8}\ps'(R_1)|V_t|\,dt+\ga^{-1}L_{\tilde{b}}\bb{E}|V_t|\,dt+dM_t\\
&=-c_1\rho(t)\,dt-\frac{L_b\ga^{-1}}{4}\ps'(R_1)|V_t|\,dt+\ga^{-1}L_{\tilde{b}}\bb{E}|V_t|\,dt+dM_t,
\end{split}\end{equation}
where $(M_t)_{t\ge0}$ is a martingale.~Here, the first inequality above follows from the definition of $\rho(t)$ (in particular, $\rho(t)\le \ps(r_s(t))$ for all $t>0$, so
$ \widetilde{\scr{L}}_{[\mu,\bar\mu],t}\,\rho(t)\le (\widetilde{\scr{L}}_{[\mu,\bar\mu],t}\,\ps)(r_s(t))$ when $\De(t)\le D_\Ga$$)$, in the second inequality we used the fact that
$\bb{W}_1(\hat{\mu}_t^X,\tilde{\mu}_t^X)\le \bb{E}|V_t|$, and in the last inequality we used \eqref{n}.

\noindent \textbf{ Case 2: $\De(t)>D_\Ga$}. In this case, $\rho(t)=\ps(D_\Ga+\ep r_l(t))$. According to the definition of the coupling operator $\widetilde{\scr{L}}_{[\mu,\bar\mu],t}$ and \eqref{ll},
\begin{align*}
dr_l(t)&=\frac{1}{2r_l(t)}\,dr_l(t)^2=\frac{1}{2r_l(t)}\left(\widetilde{\scr{L}}_{[\mu,\bar\mu],t}\,r_l(t)^2\,dt+dM_t\right)\\
&\le\frac{1}{2r_l(t)}\left(-\tau\ga r_l(t)^2\,dt+\ga^{-1}|(1-2\tau)V_t+2\ga^{-1}W_t|L_{\tilde b}(|V_t|+\bb{E}|V_t|)\,dt+dM_t\right)\\
&=-c_2r_l(t)\,dt+\frac{|(1-2\tau)V_t+2\ga^{-1}W_t|}{2\ga r_l(t)}L_{\tilde b}(|V_t|+\bb{E}|V_t|)\,dt+\frac{1}{2r_l(t)}dM_t
\end{align*}
with $c_2:={\tau\ga}/{2}$. Furthermore, by It\^{o}'s formula, $\ps'\ge0$ and $\ps''\le0$,  we have
\begin{equation}\label{o}\begin{split}
d\rho(t)&=d\ps(D_\Ga+\ep r_l(t))\le\ep\ps'(D_\Ga+\ep r_l(t))\,dr_l(t)\\
&\le\ep\ps'(D_\Ga+\ep r_l(t))\left(-c_2r_l(t)+\frac{|(1-2\tau)V_t+2\ga^{-1}W_t|}{2\ga r_l(t)}L_{\tilde b}(|V_t|+\bb{E}|V_t|)\right)\,dt+d\widetilde{M}_t,
\end{split}\end{equation}
where
\begin{equation}\label{t}
\widetilde{M}_t=\int^t_0\frac{1}{2r_l(s)}\ep\ps'(D_\Ga+\ep r_l(s))\,dM_s.
\end{equation}

For the first term of \eqref{o}, we split it into two parts
\begin{equation}\label{p}
-\frac{\ep\ps'(D_\Ga+\ep r_l(t))}{2}c_2r_l(t)\le-\ps'(2R_1)\frac{c_2\ep r_l(t)}{2(D_\Ga+\ep r_l(t))}\ps(D_\Ga+\ep r_l(t))\le-\ps'(2R_1)\frac{c_2\ep\vare}{2}\rho(t)
\end{equation}
and
\begin{equation}\label{q}
-\frac{\ep\ps'(D_\Ga+\ep r_l(t))}{2}c_2r_l(t)\le-\ps'(2R_1)\frac{c_2\ep}{2}r_l(t),
\end{equation}
where we used the facts that $\psi(r)\le r$ and $\psi'(r)\ge \psi'(2R_1)$ for all $r>0$, and noted that, due to $\De(t)>D_\Ga$,
$$
\frac{r_l(t)}{D_\Ga+\ep r_l(t)}\ge\frac{r_l(t)}{r_s(t)}\ge\vare.
$$
Hence, by \eqref{p} and \eqref{q}, we know that for the first term of \eqref{o}
\begin{equation}\label{r}
-\ep\ps'(D_\Ga+\ep r_l(t))c_2r_l(t)\le-\ps'(2R_1)\frac{c_2\ep\vare}{2}\rho(t)-\ps'(2R_1)\frac{c_2\ep}{2}r_l(t).
\end{equation}

On the other hand, for the second term of \eqref{o} we note
\begin{equation}\label{s}
\begin{split}
\ep\ps'(D_\Ga+\ep r_l(t))\frac{|(1-2\tau)V_t+2\ga^{-1}W_t|}{2\ga r_l(t)}\le&\frac{\ep}{2\ga}\sqrt{\frac{(1-2\tau)^2|V_t|^2+4(1-2\tau)\ga^{-1}\langle V_t,W_t\rangle+4\ga^{-2}|W_t|^2}{\frac{1}{2}(1-2\tau)^2|V_t|^2+(1-2\tau)\ga^{-1}\langle V_t,W_t\rangle+\ga^{-2}|W_t|^2
}}\\
\le&\frac{\ep}{\ga},\end{split}
\end{equation} where in the last inequality we used the fact that $\ps'\le 1$.
Combining \eqref{r} with \eqref{s} yields
\begin{equation}\label{u}\begin{split}
d\rho(t)
&\le-\ps'(2R_1)\frac{c_2\ep\vare}{2}\rho(t)\,dt-\ps'(2R_1)\frac{c_2\ep}{2}r_l(t)\,dt+\ep\ga^{-1}L_{\tilde b}(|V_t|+\bb{E}|V_t|)\,dt+d\widetilde{M}_t\\
&\le-\ps'(2R_1)\frac{c_2\ep\vare}{2}\rho(t)\,dt-\ps'(2R_1)\frac{c_2\ep}{4\sqrt{2}} (1-2\tau)|V_t|\,dt+\frac{1}{2}\ga^{-1}L_{\tilde b}(|V_t|+\bb{E}|V_t|)\,dt+d\widetilde{M}_t,
\end{split}
\end{equation}
where in the last inequality we used $r_l(t)\ge({1-2\tau})|V_t|/(2{\sqrt{2}})$ (see the third inequality in \eqref{e:add1} below with $\eta^2=3/4$) and $\ep\le{1}/{2}$, and $(\widetilde{M}_t)_{t\ge0}$ is given by \eqref{t}.

Combining \eqref{v} with \eqref{u} and taking the expectation yield
\begin{align*}
\frac{d}{dt}\bb{E}[\rho(t)]
&\le-\min\bigg\{c_1,\frac{c_2\ep\vare}{2}\ps'(2R_1)\bigg\}\bb{E}[\rho(t)]-\min\left\{\frac{L_b\ga^{-1}}{4}\ps'(R_1),\frac{1}{8\sqrt{2}}\ep\ga\tau(1-2\tau)\ps'(2R_1)\right\}\bb{E}|V_t|\\
&\quad+\ga^{-1}L_{\tilde b}\bb{E}|V_t| \\
&\le-\min\big\{c_1,\frac{c_2\ep\vare}{2}\ps'(2R_1)\big\}\bb{E}[\rho(t)],
\end{align*}
where in the last inequality we used \eqref{n}. By applying Gr\"{o}nwall's inequality, we get that for any $t>0$ and
$\mu_0,\bar{\mu}_0 \in \mathscr{P}_1(\R^{2d})$,
$$
\bb{W}_{\rho}({\mu}_0P^{{\mu}}_t,\bar{\mu}_0P^{\bar{\mu}}_t)\le\bb{E}[\rho(t)] \le e^{-\la t}\bb{E}[\rho(0)]
$$
with
\begin{equation}\label{w}
\la=\min\Big\{\al\ga\ps'(R_1),\al\ga\frac{k_0-4}{4(k_0\al+1)}\ps'(R_1),\frac{\tau\gamma\ep\vare}{4}\ps'(2R_1)\Big\}.
\end{equation}

Taking the infimum over all couplings $\Pi$ of
${\mu}_0$ and $\bar{\mu}_0$, we obtain \eqref{Wde}, i.e.,
 $$\bb{W}_{\rho}({\mu}_0P^{{\mu}}_t,\bar{\mu}_0P^{\bar{\mu}}_t)\le e^{-\la t}\bb{W}_{\rho}({\mu}_0,\bar{\mu}_0).$$

Next, we prove the desired assertion.
Recalling that
${\mu}_0P^{{\mu}}_t=\mu_{t}$ and $\bar{\mu}_0P^{\bar{\mu}}_t=\bar\mu_{t}$ when ${\mu}_0=\mu$ and $\bar{\mu}_0=\bar\mu$ respectively, we have
$$
 \bb{W}_{\rho}(\mu_t,\bar\mu_t)\le e^{-\la t}\bb{W}_{\rho}(\mu,\bar\mu).
$$
Therefore, we obtain, by Lemma \ref{lemma3.4} below, for all $t>0$,
$$
 \bb{W}_1(\mu_t,\bar\mu_t)\le \frac{1}{ M_1}\bb{W}_{\rho}(\mu_t,\bar\mu_t)\le \frac{1}{ M_1}e^{-\la t}\bb{W}_{\rho}(\mu,\bar\mu)\le \frac{M_2}{M_1}e^{-\la t}\bb{W}_1(\mu,\bar\mu),
$$
where $M_1$ and $M_2$ are given in Lemma \ref{lemma3.4}.
Hence, the desired assertion follows with $C_1=M_2/M_1$ and $\lambda$ given by \eqref{w}.
\end{proof}

\begin{lemma}\label{lemma3.4}
The metric $\rho$ given by \eqref{Metric} is equivalent to the Euclidean distance on $\R^{2d}$. More explicitly,  for any $x,y,x',y'\in \R^d$,
\begin{equation}\label{equivalence}
M_1|(x,y)-(x',y')|\le\rho((x,y),(x',y'))\le M_2|(x,y)-(x',y')|,
\end{equation}
where $M_1=\ep \ps'(2R_1)\min\Big\{\frac{1-2\tau}{2\sqrt{2}},\frac{\ga^{-1}}{\sqrt{3}}\Big\}$, and $M_2=\sqrt{2}\max\{\al+1,\ga^{-1}\}$.
\end{lemma}

\begin{proof}
Recall that $r_s((x,y),(x',y')):=\al|x-x'|+|x-x'+\ga^{-1}(y-y')|$, and $r_l((x,y),(x',y'))$ is given by \eqref{rl}. On the one hand, for all $(x,y),\,(x',y')\in\R^{2d}$,
\begin{align*}
\De((x,y),(x',y'))\wedge D_\Ga+\ep r_l((x,y),(x',y'))
&=r_s((x,y),(x',y'))\wedge (D_\Ga+\ep r_l((x,y),(x',y')))\\
&\le r_s((x,y),(x',y'))\\
&\le\max\{\al+1,\ga^{-1}\}(|x-x'|+|y-y'|)\\
&\le\sqrt{2}\max\{\al+1,\ga^{-1}\}(|x-x'|^2+|y-y'|^2)^{{1}/{2}},
\end{align*}
where we used $(a+b)^2\le2(a^2+b^2)$ in the last inequality.

On the other hand, for all $(x,y),\,(x',y')\in\R^{2d}$,
\begin{equation}\label{e:add1}\begin{split}
&\De((x,y),(x',y'))\wedge D_\Ga+\ep r_l((x,y),(x',y'))\\
&\ge\ep r_l((x,y),(x',y')) =\ep  \left( \frac{1}{2}(1-2\tau)^2|v|^2+(1-2\tau)\ga^{-1}\langle v,w\rangle+\ga^{-2}|w|^2\right)^{1/2}\\
&=\ep  \bigg( \frac{1-\eta^2}{2}(1-2\tau)^2|v|^2+\frac{\eta^2}{2}(1-2\tau)^2|v|^2+\eta(1-2\tau)\eta^{-1}\ga^{-1}\langle v,w\rangle\\
&\qquad +\frac{1}{2}\eta^{-2}\ga^{-2}|w|^2+\Big(1-\frac{\eta^{-2}}{2}\Big)\ga^{-2}|w|^2\bigg)^{1/2}\\
&\ge\frac{\ep}{\sqrt{2}}\Big[(1-\eta^2)(1-2\tau)^2|x-x'|^2+(2-\eta^{-2})\ga^{-2}|y-y'|^2\Big]^{1/2}\\
&\ge\ep\min\Big\{\frac{(1-\eta^2)^{1/2}(1-2\tau)}{\sqrt{2}},\frac{\ga^{-1}(2-\eta^{-2})^{1/2}}{\sqrt{2}}\Big\}(|x-x'|^2+|y-y'|^2)^{1/2},
\end{split}\end{equation} where $\eta\in (\sqrt{2}/2,1).$ In particular, taking $\eta^2=3/4$ we arrive at that
$$\De((x,y),(x',y'))\wedge D_\Ga+\ep r_l((x,y),(x',y'))\ge\ep\min\Big\{\frac{1-2\tau}{2\sqrt{2}},\frac{\ga^{-1}}{\sqrt{3}}\Big\}(|x-x'|^2+|y-y'|^2)^{1/2}.$$

Combining both conclusions above with  \eqref{inequality} yields the desired assertion.
\end{proof}

\section{Proof of Theorem \ref{chaos}}\label{section4}
This section is devoted to the proof of Theorem \ref{chaos}, which shows the propagation of chaos (uniformly in time) of the mean-field interacting particle system \eqref{meanfield} towards the McKean-Vlasov type Langevin dynamic \eqref{Equation} with L\'evy noises.

Fix $N\in\NN$. We define the metric $\rho_N$: $\RR^{2Nd}\times\RR^{2Nd}\to[0,\infty)$ by
\begin{equation}\label{NMetric}
\rho_N((x,y),(x',y')):=N^{-1}\sum_{i=1}^N\rho((x^i,y^i),(x'^i,y'^i)),\quad ((x,y),(x',y'))\in\RR^{2Nd}\times\RR^{2Nd},
\end{equation}
where $x=(x^1,x^2,\cdots, x^N)\in \RR^{Nd}$, and $\rho$ is given by \eqref{Metric}. By \eqref{equivalence}, $\rho_N$ is equivalent to $l^1_N$ given by \eqref{lN} so that for all $((x,y),(x',y'))\in\RR^{2Nd}\times\RR^{2Nd}$,
\begin{equation}\label{equivalenceN}
M_1\,l^1_N((x,y),(x',y'))\le\rho_N((x,y),(x',y'))\le M_2\,l^1_N((x,y),(x',y')).
\end{equation} where $M_1$ and $M_2$ are given  in Lemma \ref{lemma3.4}.

In the following, let $\{(X_t^i,Y_t^i)_{t\ge0}\}_{1\le i\le N}$ be $N$-independent versions of the solution to \eqref{Equation}. Then, for each $1\le i\le N$,  $(X_t^i,Y_t^i)_{t\ge0}$ shares the same distribution  as $(\mu_t)_{t\ge0}$ for the solution to \eqref{Equation}. Furthermore, we consider the corresponding decoupled SDE
\begin{equation}\left\{\begin{array}{l}\label{decoupledN}
dX^{i,\mu}_t=Y^{i,\mu}_tdt,\\
dY^{i,\mu}_t=\Big(b(X^{i,\mu}_t)+\displaystyle\int_{\R^d}\tilde{b}(X^{i,\mu}_t,z)\,\mu_t^{X}(dz)-\ga Y^{i,\mu}_t\Big)\,dt+dL_t^i,\quad i=1,2,\cdots,N,\end{array}\right.
\end{equation}
where for every $t>0$, $\mu_t^{X}={\rm Law}(X_t^i)$ (that is same for all $1\le i\le N$).
Due to the strong well-posedness (so the weak well-posedness) of \eqref{Equation}, we obtain that for
$\mu={\rm Law}((X^i_0,Y^i_0))$ and $\mu_0={\rm Law}((X^{i,\mu}_0,Y^{i,\mu}_0))$,
\begin{equation}\label{weak}
 \mu_0 P_t^{\mu}=\mu_t \quad {\rm if}~~\mu_0=\mu,
\end{equation} where, for every $t>0$, $\mu_0P_t^{\mu}$ is the law of $(X^{i,\mu}_t, Y^{i,\mu}_t)$ corresponding to the $i$-th component of the SDE \eqref{decoupledN}.
Therefore, in order to obtain Theorem $\ref{chaos}$, thanks to Remark \ref{Remark2} (i), it suffices to quantify the propagation of chaos of the particle system \eqref{meanfield} towards the decoupled SDE \eqref{decoupledN} in the following sense
\begin{equation}\label{WNde}
\bb W_{\rho_N}((\mu_0 P_t^{\mu})^{\otimes N},\bar\mu^N_t)\le e^{-\lambda t}\bb W_{\rho_N}({\mu}_0^{\otimes N},\bar\mu^{\otimes N})+C_0 N^{-1/2},\quad t>0,
\end{equation}
where
$(\mu_0 P_t^{\mu})^{\otimes N}$ is the law of the solution to \eqref{decoupledN} with initial distribution
 $\mu_0^{\otimes N}$,
  $\bar\mu^N_t$ is the law of particles driven by  \eqref{meanfield} with initial distribution $\bar\mu^{\otimes N}$, and $\bb W_{\rho_N}$ is the $L^1$-Wasserstein distance with respect to the metric $\rho_N$ given by \eqref{NMetric}.

\subsection{Coupling process for the particle system \eqref{meanfield} and the decoupled SDE \eqref{decoupledN}}
Fix $N\in\NN$. Note that under Assumptions {\rm \textbf{(A1)}} and {\rm\textbf{(A2)}}, the SDE \eqref{decoupledN} has a unique strong solution $\{(X^{i,\mu}_t,Y^{i,\mu}_t)_{t\ge0}\}_{1\le i\le N}$ for any initial condition. We will construct a Markov coupling process
of $\{(X^{i,\mu}_t,Y^{i,\mu}_t)_{t\ge0}\}_{1\le i\le N}$  and $\{(\bar{X}^{i,N}_t,\bar{Y}^{i,N}_t)_{t\ge0})\}_{1\le i\le N}$ solving \eqref{decoupledN} and  \eqref{meanfield} respectively. For this, we turn to an associated coupling operator that is motivated by that in Subsection \ref{sub coupling operator}.

For any $x,y\in \RR^{Nd}$ and $\mu\in\scr{P}_1(\R^{2d})$, the generator of the process $\{(X^{i,\mu}_t,Y^{i,\mu}_t)_{t\ge0}\}_{1\le i\le N}$, acting on $f\in C^2_b(\R^{2Nd})$, is given by $$(\scr{L}_{[{\mu}],t}f)(x,y) =\sum_{i=1}^N(\scr{L}^i_{[{\mu}],t}f^i)(x^i,y^i),$$ where $f^i(x^i,y^i)=f((x^1,\cdots, x^i,\cdots,x^N),(y^1,\cdots,
y^i,\cdots,y^N))$ and
\begin{equation}\label{Li}
\begin{split}
(\scr{L}^i_{[{\mu}],t}f^i)(x^i,y^i)
&=\big\langle y^i, \nabla_{x^i}f^i(x^i,y^i)\big\rangle+\bigg\langle b(x^i)+\displaystyle\int_{\R^d}\tilde{b}(x^i,z)\, {\mu}^X_t(dz)-\ga y^i, \nabla_{y^i}f^i(x^i,y^i)\bigg\rangle\\
&\quad+\int_{\R^d}\big(f^i(x^i,y^i+z)-f^i(x^i,y^i)-\langle\nabla_{y^i}f^i(x^i,y^i),z\rangle\I_{\{|z|\le1\}}\big)\,\nu(dz);
\end{split}
\end{equation}
while the generator of
$\{(\bar{X}^{i,N}_t,\bar{Y}^{i,N}_t)_{t\ge0}\}_{1\le i\le N}$ is given by
$$(\scr{L}_{N,t}f)(x,y) =\sum_{i=1}^N(\scr{L}^i_{N,t}f^i)(x^i,y^i),$$ where
\begin{equation}\label{LN}
\begin{split}
(\scr{L}^i_{N,t}f^i)(x^i,y^i)
&=\big\langle y^i, \nabla_{x^i}f^i(x^i,y^i)\big\rangle+\bigg\langle b(x^i)+N^{-1}\sum_{j=1}^N\tilde{b}(x^i,x^j)-\ga y^i, \nabla_{y^i}f^i(x^i,y^i)\bigg\rangle\\
&\quad+\int_{\R^d}\big(f^i(x^i,y^i+z)-f^i(x^i,y^i)-\langle\nabla_{y^i}f^i(x^i,y^i),z\rangle\I_{\{|z|\le1\}}\big)\,\nu(dz).
\end{split}
\end{equation}

Hence, for any fixed $\mu\in\scr P_1(\RR^{2d})$, $1\le i\le N$, $x^i,y^i,x'^i,y'^i\in\RR^d$ and $f\in C_b^2(\RR^{4d})$, define
\begin{equation}\label{iCoupling operator}
\begin{split}
&\big(\widetilde{\scr{L}}_{[\mu,N],t}^i f\big)\big((x^i,y^i),(x'^i,y'^i)\big)\\
&=\big(\widetilde{\scr{L}}^i_{[\mu,N],t,0}f\big)\big((x^i,y^i),(x'^i,y'^i)\big)+\big(\widetilde{\scr{L}}_{x^i,x'^i}f\big)\big((x^i,\cdot),(x'^i,\cdot)\big)(y^i,y'^i)\I_{\{\De((x^i,y^i),(x'^i,y'^i))\le D_\Ga\}} \\
&\quad+\big(\widetilde{\scr{L}}^\ast_{x^i,x'^i}f\big)\big((x^i,\cdot),(x'^i,\cdot)\big)(y^i,y'^i)\I_{\{\De((x^i,y^i),(x'^i,y'^i))>D_\Ga\}},
\end{split}\end{equation}
where $\widetilde{\scr{L}}_{x^i,x'^i}$ and $\widetilde{\scr{L}}^*_{x^i,x'^i}$ are given by \eqref{Lxx'} and \eqref{Lxx''} respectively, and
\begin{align*}
&\big(\widetilde{\scr{L}}^i_{[\mu,N],t,0}f\big)\big((x^i,y^i),(x'^i,y'^i)\big)\nonumber\\
&:=\bigg\langle b(x^i)+\int_{\R^d}\widetilde{b}(x^i,z)\,{\mu}_t^X(dz)-\ga y^i, \nabla_{y^i}f\big((x^i,y^i),(x'^i,y'^i)\big)\bigg\rangle+\big\langle y^i, \nabla_{x^i}f\big((x^i,y^i),(x'^i,y'^i)\big)\big\rangle\\
&\quad\,\,\,+\bigg\langle b(x'^i)+N^{-1}\sum_{j=1}^N\tilde{b}(x'^i,x'^j)-\ga y'^i, \nabla_{y'^i}f\big((x^i,y^i),(x'^i,y'^i)\big)\bigg\rangle+\big\langle y'^i, \nabla_{x'^i}f\big((x^i,y^i),(x'^i,y'^i)\big)\big\rangle.\nonumber
\end{align*}
According to Proposition \ref{b}, we can verify that, for any $1\le i \le N$, $\widetilde{\scr{L}}_{[\mu,N],t}^i$ given by \eqref{iCoupling operator} is a coupling operator of $\scr{L}_{[\mu],t}^i$ and $\scr{L}_{N,t}^i$ given by \eqref{Li} and \eqref{LN} respectively. Hence, $\sum_{i=1}^N\widetilde{\scr{L}}_{[\mu,N],t}^i$ is a coupling operator of the operators $\scr{L}_{[{\mu}],t}$ and $\scr{L}_{N,t}$.

Now, for $1\le i\le N$,  consider the following SDE
\begin{equation}\left\{\begin{array}{l}\label{iCoupling process}
dX^{i,\mu}_t=Y^{i,\mu}_tdt,\\
dY^{i,\mu}_t=\left(b(X^{i,\mu}_t)+\displaystyle\int_{\R^d}\tilde{b}(X^{i,\mu}_t,z)\,\mu_t^{X}(dz)-\ga Y^{i,\mu}_t\right)\,dt+dL_t^i,~~~\mu_t^{X}={\rm Law}(X_t^i)\\
d\bar{X}^{i,N}_t=\bar{Y}^{i,N}_tdt,\\
d\bar{Y}^{i,N}_t=\left(b(\bar{X}^{i,N}_t)+N^{-1}\sum_{j=1}^N\tilde{b}(\bar{X}^{i,N}_t,\bar{X}^{j,N}_t)-\ga \bar{Y}^{i,N}_t\right)\,dt+dL^{\ast,i}_t. \end{array}\right.
\end{equation}
Herein,
$$
dL^{\ast,i}_t=dL_t^i+\I_{\{\De((X^{i,\mu},Y^{i,\mu}_t),(\bar{X}^{i,N}_t,\bar{Y}^{i,N}_t))\le D_\Ga\}} \ga(Q^i_{t-})_\ka\int_{\R^d\times[0,1]}\Xi(\ga(Q^i_{t-})_\ka,z,u)N(dt,dz,du),
$$
where $\Xi$ is given by \eqref{ee},
$$
Q^i_{t}:=V^i_{t}+\ga^{-1}W^i_t \quad \hbox{with} \quad V^i_t:=X^{i,\mu}_t-\bar{X}^{i,N}_t \quad \hbox{and} \quad W^i_t:=Y^{i,\mu}_t-\bar{Y}^{i,N}_t.
$$
Furthermore, one can follow the proof of \cite[Proposition 2.2]{LW} and obtain that \eqref{iCoupling process} has a unique strong solution  $\{((X^{i,\mu}_t,Y^{i,\mu}_t),(\bar{X}^{i,N}_t,\bar{Y}^{i,N}_t))_{t\ge0}\}_{1\le i\le N}$.
Repeating the proof of Proposition \ref{coupling proof}, we can find that the infinitesimal generator of the process $\{((X^{i,\mu}_t,Y^{i,\mu}_t),(\bar{X}^{i,N}_t,\bar{Y}^{i,N}_t))_{t\ge0}\}_{1\le i\le N}$ is just the operator
 $\sum_{i=1}^N\widetilde{\scr{L}}_{[\mu,N],t}^i$, which is a coupling of the operators $\scr{L}_{[{\mu}],t}$ and $\scr{L}_{N,t}$. Therefore,
$\{((X^{i,\mu}_t,Y^{i,\mu}_t),(\bar{X}^{i,N}_t,\bar{Y}^{i,N}_t))_{t\ge0}\}_{1\le i\le N}$
is a coupling process of $\{(X^{i,\mu}_t,Y^{i,\mu}_t)_{t\ge0}\}_{1\le i\le N}$  and $\{(\bar{X}^{i,N}_t,\bar{Y}^{i,N}_t)_{t\ge0})\}_{1\le i\le N}$.

\subsection{Proof of Theorem \ref{chaos}}
In this subsection, for any $x^i,x'^i,y^i,y'^i\in\R^d$, we use notations of $v^i,w^i,q^i,r^i_s,r^i_l,\rho^i$ as the same way as these in Section \ref{Proofs}.

Recall that $\ps$ is the function defined by \eqref{A}. The following lemma is similar to Lemma \ref{le1}, and so we omit the proof here.

\begin{lemma}
For each $1\le i\le N$ and the
 operator $\widetilde{\scr{L}}_{[\mu,N],t}^i$ given by \eqref{iCoupling operator}, it holds that
\begin{align*}
(\widetilde{\scr{L}}_{[\mu,N],t}^i\,\ps)(r^i_s)&=\ps'(r^i_s)\bigg\{\frac{1}{|q^i|}\bigg\langle q^i,\ga^{-1}\big(b(x^i)-b(x'^i)+\int_{\R^d}\tilde{b}(x^i,z)\,\mu^X_t(dz)-N^{-1}\sum_{j=1}^N\tilde{b}(x'^i,x'^j)\big)\bigg\rangle\nonumber\\
&\qquad\qquad\,\, -\al\ga|v^i|+\frac{\al\ga}{|v^i|}\langle v^i,q^i\rangle\bigg\}\nonumber\\
&\quad+\frac{1}{2}\big(\ps(r^i_s-(\ka\wedge|q^i|))+\ps(r^i_s+(\ka\wedge|q^i|))-2\ps(r^i_s)\big)\nu_{\ga(q^i)_\ka}(\R^d)\I_{\{\Delta((x^i,y^i),(x'^i,y'^i))\le D_\Gamma\}}.\end{align*}
\end{lemma}

Then we have the following local contraction statement.
\begin{proposition}\label{Proposition3} Suppose that Assumptions {\rm\textbf{(A0)}}--{\rm\textbf{(A2)}} and \eqref{hh} hold. For each $1\le i\le N$, set
$$ a_t^i :=\bigg|\int_{\R^d}\tilde{b}(x^i,z)\,\mu^X_t(dz)-N^{-1}\sum_{j=1}^N\tilde{b}(x^i,x^j)\bigg|.$$ Then the following holds.
\begin{itemize}
\item[{\rm (i)}] It holds that for any $x^i, x'^i, y^i, y'^i\in\R^d$ with $\De((x^i,y^i),(x'^i,y'^i))\le D_\Ga$,
\begin{equation}\label{Lci}
\widetilde{\scr{L}}_{[\mu,N],t}^i\,\ps(r^i_s)\le-c_1\ps(r^i_s)-\frac{\al\ga}{4}\ps'(R_1)|v^i|+\ga^{-1}\bigg(N^{-1}\sum_{j=1}^N L_{\tilde{b}}(|v^i|+|v^j|)+ a_t^i \bigg),
\end{equation}
where $c_1$ is given by \eqref{c1}.

\item[{\rm(ii)}] It holds that for any $x^i, x'^i, y^i, y'^i\in\R^d$ with $\De((x^i,y^i),(x'^i,y'^i))>D_\Ga$,
\begin{equation}\label{lli}
\widetilde{\scr{L}}_{[\mu,N],t}^i\,(r^i_l)^2 \le-\tau\ga (r^i_l)^2+\ga^{-1}|(1-2\tau)v^i+2\ga^{-1}w^i|\bigg(N^{-1}\sum_{j=1}^N L_{\tilde{b}}(|v^i|+|v^j|)+ a_t^i \bigg),
\end{equation}
where $\tau$ is given by \eqref{tau}.
\end{itemize}
\end{proposition}

\begin{proof}
The proof mainly follows
those of Propositions \ref{Proposition1} and \ref{Proposition2}, and we only need to handle the term $$\int_{\R^d}\tilde{b}(x^i,z)\,\mu_t^X(dz)-N^{-1}\sum_{j=1}^N\tilde{b}(x'^i,x'^j)$$ involved in place of
$$\int_{\R^d}\tilde{b}(x,z)\,\hat{\mu}^X_t(dz)-\int_{\R^d}\tilde{b}(x',z')\,\tilde{\mu}^X_t(dz').$$

Indeed, by Assumption \textbf{(A2)},
\begin{align*}
&\bigg|\int_{\R^d}\tilde{b}(x^i,z)\,\mu_t^X(dz)-N^{-1}\sum_{j=1}^N\tilde{b}(x'^i,x'^j)\bigg|\\
&\le \bigg|\int_{\R^d}\tilde{b}(x^i,z)\,\mu_t^X(dz)-N^{-1}\sum_{j=1}^N\tilde{b}(x^i,x^j)\bigg|+N^{-1}\sum_{j=1}^N\bigg|\tilde{b}(x'^i,x'^j)-\tilde{b}(x^i,x^j)\bigg|\\
&\le  a_t^i +N^{-1}\sum_{j=1}^N L_{\tilde{b}}(|v^i|+|v^j|).
\end{align*}
With this at hand, we can repeat the proofs of Propositions \ref{Proposition1} and \ref{Proposition2} to obtain the desired assertion.
\end{proof}

To prove Theorem \ref{chaos}, we consider the coupling process $\{((X^{i,\mu}_t,Y^{i,\mu}_t),(\bar{X}^{i,N}_t,\bar{Y}^{i,N}_t))_{t\ge0}\}_{1\le i\le N}$
given by \eqref{iCoupling process} and apply Proposition \ref{Proposition3}.
For simplicity, for any
$t\ge0$ and $1\le i\le N$, we write $r^i_s(t)=r_s((X^{i,\mu}_t,Y^{i,\mu}_t),(\bar{X}^{i,N}_t,\bar{Y}^{i,N}_t))$, $r^i_l(t)=r_l((X^{i,\mu}_t,Y^{i,\mu}_t),(\bar{X}^{i,N}_t,\bar{Y}^{i,N}_t))$,
$\Delta^i(t)=\Delta((X^{i,\mu}_t,Y^{i,\mu}_t),(\bar{X}^{i,N}_t,\bar{Y}^{i,N}_t))$,
$\rho^i(t)=\psi((\De^i(t)\wedge D_\Ga)+\ep r^i_l(t))$
and $\rho_N(t)=N^{-1}\sum_{i=1}^N\rho^i(t)$.

\begin{proof}[Proof of Theorem $\ref{chaos}$]
First, for each $1\le i\le N$, set
$$
A^i_t:=\bigg|\int_{\R^d}\tilde{b}(X_t^{i,\mu},z)\,\mu^X_t(dz)-N^{-1}\sum_{j=1}^N\tilde{b}(X_t^{i,\mu},X_t^{j,\mu})\bigg|.
$$
With the aid of Proposition \ref{Proposition3}, one can follow the proof of Theorem  \ref{Main theorem} to see that
\begin{itemize}
\item  When $\De^i(t)\le D_\Ga$,
\begin{equation}\label{vi}
d\rho^i(t)\le-c_1\rho^i(t)\,dt-\frac{L_b\ga^{-1}}{4}\ps'(R_1)|V^i_t|\,dt+\ga^{-1}\bigg(N^{-1}\sum_{j=1}^N L_{\tilde{b}}|V^j_t|+A^i_t\bigg)\,dt+dM^i_t,
\end{equation}
where
$(M^i_t)_{t\ge0}$ is a martingale.
\item When $\De^i(t)>D_\Ga$,
\begin{equation}\label{ui}\begin{split}
d\rho^i(t)
&\le-\ps'(2R_1)\frac{c_2\ep\vare}{2}\rho^i(t)\,dt-\ps'(2R_1)\frac{c_2\ep}{4\sqrt{2}} (1-2\tau)|V^i_t|\,dt\\
&\quad+\frac{1}{2}\ga^{-1}\bigg(N^{-1}\sum_{j=1}^N L_{\tilde{b}}(|V^i_t|+|V^j_t|)+A^i_t\bigg)\,dt+d\widetilde{M}^i_t,
\end{split}
\end{equation} where $(\widetilde{M}^i_t)_{t\ge0}$ is a martingale.
\end{itemize}

Combining \eqref{vi} with \eqref{ui}, taking the expectation and summing over $i=1,2,\cdots,N$ yield
\begin{align*}
\frac{d}{dt}\bb{E}\bigg[N^{-1}\sum_{i=1}^N\rho^i(t)\bigg]
&\le-\min\bigg\{c_1,\frac{c_2\ep\vare}{2}\ps'(2R_1)\bigg\}\bb{E}\bigg[N^{-1}\sum_{i=1}^N\rho^i(t)\bigg]+\ga^{-1}\bb{E}\bigg[N^{-1}\sum_{i=1}^NA^i_t\bigg]\nonumber\\
&\quad-\min\left\{\frac{L_b\ga^{-1}}{4}\ps'(R_1),\frac{\ep\ga}{8\sqrt{2}}\tau(1-2\tau)\ps'(2R_1)\right\}\bb{E}\bigg[N^{-1}\sum_{i=1}^N|V^i_t|\bigg]\\
&\quad +\ga^{-1}L_{\tilde b}\bb{E}\bigg[N^{-1}\sum_{i=1}^N|V^i_t|\bigg]\nonumber\\
&\le-\la\,\bb{E}\bigg[N^{-1}\sum_{i=1}^N\rho^i(t)\bigg]+\ga^{-1}\bb{E}\bigg[N^{-1}\sum_{i=1}^NA^i_t\bigg],
\end{align*}
where $\la$ is given by \eqref{w} and in the second inequality we used again \eqref{n}.

According to Lemmas \ref{reAi} and \ref{C2} below, we find that there exists a constant $C'>0$ such that
$$
\sup_{1\le i\le N}\sup_{t\ge0}\E A^i_t\le C'N^{-1/2}.
$$
Inserting this bound of $\E A^i_t$ into the inequality above yields
$$
\frac{d}{dt}\bb{E}\big[\rho_N(t)\big]\le-\la\,\bb{E}\big[\rho_N(t)\big]+C'\ga^{-1} N^{-1/2}=:-\la\,\bb{E}\big[\rho_N(t)\big]+C'' N^{-1/2}.
$$
Applying Gr\"{o}nwall's inequality, we find that for any $t>0$ and
$\mu_0\in \mathscr{P}_1(\R^{2d})$,
$$
\bb{W}_{\rho_N}(({\mu}_0P^{{\mu}}_t)^{\otimes N},\bar{\mu}^N_t)
\le\bb{E}\big[\rho_N(t)\big]\le e^{-\la t}\,\bb{E}\big[\rho_N(0)\big]+\la ^{-1}C''N^{-1/2}=:  e^{-\la t}\,\bb{E}\big[\rho_N(0)\big]+ C_0 N^{-1/2}.
$$
Thus, taking the infimum over all couplings of $\mu_0^{\otimes N}$ and $\bar{\mu}^{\otimes N}$, we obtain \eqref{WNde}, i.e.,
$$
\bb{W}_{\rho_N}(({\mu}_0P^{{\mu}}_t)^{\otimes N},\bar{\mu}^N_t)\le e^{-\la t}\,\bb{W}_{\rho_N}(\mu_0^{\otimes N},\bar{\mu}^{\otimes N})+C_0N^{-1/2}.
$$

Recalling that
${\mu}_0P^{{\mu}}_t=\mu_{t}$ when ${\mu}_0=\mu$, we have
$$
\bb{W}_{\rho_N}(\mu_t^{\otimes N},\bar\mu_t^N)\le e^{-\la t}\bb{W}_{\rho_N}(\mu^{\otimes N},\bar\mu^{\otimes N})+C_0N^{-1/2}.
$$
Therefore, we obtain, by \eqref{equivalenceN}, for all $t>0$,
$$
\bb{W}_{l_N^1}(\mu_t^{\otimes N},\bar\mu_t^N)
\le
\frac{M_2}{M_1}e^{-\la t}\bb{W}_{l_N^1}(\mu^{\otimes N},\bar\mu^{\otimes N})+\frac{C_0}{M_1}N^{-1/2},
$$
where $M_1$ and $M_2$ are given  in Lemma \ref{lemma3.4}.
Hence, the desired assertion follows.
\end{proof}

\begin{lemma}\label{reAi} For all $1\le i\le N$ and $t>0$,
\begin{equation}\label{Ai}
\E A^i_t\le L_{\tilde{b}}\bigg(\frac{\sqrt{2}}{\sqrt{N}}+\frac{2}{N}\bigg)\bigg(\int_{\R^{d}}|z|^2\,\mu^X_t(dz)\bigg)^{1/2}.
\end{equation}\end{lemma}
\begin{proof}
First, we note that the random variables $X^{i,\mu}_t$, $1\le i \le N$, are independent and identically distributed, and that, due to \eqref{weak}, for any $1\le i\le N$, $\scr L_{X^{i,\mu}_t}=\mu_t^X$ if $\mu_0=\mu$. Thus,
\begin{equation}\label{conexpec}
\E\big(\tilde{b}(X_t^{i,\mu},X_t^{j,\mu})|X_t^{i,\mu}\big)=\int_{\R^d}\tilde{b}(X_t^{i,\mu},z)\mu_t^X(dz),
\end{equation}
and so \begin{align*}
\E A^i_t
&=\E\bigg|\int_{\R^d}\tilde{b}(X_t^{i,\mu},z)\,\mu^X_t(dz)-N^{-1}\sum_{j=1}^N\tilde{b}(X_t^{i,\mu},X_t^{j,\mu})\bigg|\\
&\le\frac{1}{N}\E\bigg|(N-1)\int_{\R^{d}}\tilde{b}(X_t^{i,\mu},z)\,\mu^X_t(dz)-\sum^{N}_{j=1:j\neq i}\tilde{b}(X_t^{i,\mu},X_t^{j,\mu})\bigg|\\
&\quad+\frac{1}{N}\E\bigg|\int_{\R^{d}}\tilde{b}(X_t^{i,\mu},z)\,\mu^X_t(dz)-\tilde{b}(X_t^{i,\mu},X_t^{i,\mu})\bigg|\\
&\le\frac{1}{N}\left(\E\bigg(\E\bigg(\bigg|(N-1)\int_{\R^{d}}\tilde{b}(X_t^{i,\mu},z)\,\mu^X_t(dz)-\sum^{N}_{j=1:j\neq i}\tilde{b}(X_t^{i,\mu},X_t^{j,\mu})\bigg|^2\bigg|X_t^{i,\mu}\bigg)\bigg)\right)^{1/2}\\
&\quad+\frac{1}{N}\E\left(\int_{\R^{d}}\big|\tilde{b}(X_t^{i,\mu},z)-\tilde{b}(X_t^{i,\mu},X_t^{i,\mu})\big|\,\mu^X_t(dz)\right)\\
&=:I_1+I_2.
\end{align*}

For $I_1$, set
\begin{align*}
I_3:
&=\bigg|(N-1)\int_{\R^{d}}\tilde{b}(X_t^{i,\mu},z)\,\mu^X_t(dz)-\sum^{N}_{j=1:j\neq i}\tilde{b}(X_t^{i,\mu},X_t^{j,\mu})\bigg|^2\\
&=(N-1)^2\bigg(\int_{\R^{d}}\tilde{b}(X_t^{i,\mu},z)\,\mu^X_t(dz)\bigg)^2-2(N-1)\int_{\R^{d}}\tilde{b}(X_t^{i,\mu},z)\,\mu^X_t(dz)\cdot\sum^{N}_{j=1:j\neq i}\tilde{b}(X_t^{i,\mu},X_t^{j,\mu})\\
&\quad+\bigg(\sum^{N}_{j=1:j\neq i}\tilde{b}(X_t^{i,\mu},X_t^{j,\mu})\bigg)^2.
\end{align*}
Then, by \eqref{conexpec},
\begin{align*}
\E(I_3|X_t^{i,\mu})
&=(N-1)^2\bigg(\int_{\R^{d}}\tilde{b}(X_t^{i,\mu},z)\,\mu^X_t(dz)\bigg)^2-2(N-1)^2\bigg(\int_{\R^{d}}\tilde{b}(X_t^{i,\mu},z)\,\mu^X_t(dz)\bigg)^2\\
&\quad+\sum^{N}_{j=1:j\neq i}\sum^{N}_{k=1:k\neq i}\E\big(\tilde{b}(X_t^{i,\mu},X_t^{j,\mu})\cdot\tilde{b}(X_t^{i,\mu},X_t^{k,\mu})\big|X_t^{i,\mu}\big)\\
&=-(N-1)^2\bigg(\int_{\R^{d}}\tilde{b}(X_t^{i,\mu},z)\,\mu^X_t(dz)\bigg)^2+(N-1)(N-2)\bigg(\int_{\R^{d}}\tilde{b}(X_t^{i,\mu},z)\,\mu^X_t(dz)\bigg)^2\\
&\quad+(N-1)\int_{\R^{d}}\big(\tilde{b}(X_t^{i,\mu},z)\big)^2\,\mu^X_t(dz)\\
&=(N-1)\,{\rm Var}_{\mu_t^X}\big(\tilde{b}(X_t^{i,\mu},\cdot)\big).
\end{align*}
Here in the second equality above we note that when $j\neq k$, $X^{j,\mu}_t$ and $X^{k,\mu}_t$ are independent, and so, by \eqref{conexpec} again,
\begin{align*}
\E\big(\tilde{b}(X_t^{i,\mu},X_t^{j,\mu})\cdot\tilde{b}(X_t^{i,\mu},X_t^{k,\mu})\big|X_t^{i,\mu}\big)
&=\E\big(\tilde{b}(X_t^{i,\mu},X_t^{j,\mu})\big|X_t^{i,\mu}\big)\cdot\E\big(\tilde{b}(X_t^{i,\mu},X_t^{k,\mu})\big|X_t^{i,\mu}\big)\\
&=\bigg(\int_{\R^{d}}\tilde{b}(X_t^{i,\mu},z)\,\mu^X_t(dz)\bigg)^2;
\end{align*}
that when $j=k$,
$$\E\big(\tilde{b}(X_t^{i,\mu},X_t^{j,\mu})\cdot\tilde{b}(X_t^{i,\mu},X_t^{k,\mu})\big|X_t^{i,\mu}\big)=\int_{\R^{d}}\big(\tilde{b}(X_t^{i,\mu},z)\big)^2\,\mu^X_t(dz).$$
Furthermore, due to Assumption {\rm\textbf{(A2)}},
\begin{align*}
{\rm Var}_{\mu_t^X}\big(\tilde{b}(X_t^{i,\mu},\cdot)\big)
&=\frac{1}{2}\int_{\R^{d}}\int_{\R^{d}}\big|\tilde{b}(X_t^{i,\mu},z)-\tilde{b}(X_t^{i,\mu},z')\big|^2\,\mu^X_t(dz)\mu^X_t(dz')\\
&\le\frac{1}{2}L^2_{\tilde{b}}\int_{\R^{d}}\int_{\R^{d}}|z-z'|^2\,\mu^X_t(dz)\mu^X_t(dz')\\
&\le2L^2_{\tilde{b}}\int_{\R^{d}}|z|^2\,\mu^X_t(dz).
\end{align*}
Here, we obtain
\begin{align*}
I_1&=\frac{1}{N}\bigg(\E\big(\E(I_3|X_t^{i,\mu})\big)\bigg)^{1/2}
\le \frac{1}{N}\bigg(2(N-1)L_{\tilde{b}}^2\int_{\R^{d}}|z|^2\,\mu^X_t(dz)\bigg)^{1/2}\\
&\le \frac{\sqrt{2}L_{\tilde{b}}}{\sqrt{N}}\bigg(\int_{\R^{d}}|z|^2\,\mu^X_t(dz)\bigg)^{1/2}.
\end{align*}

For $I_2$, by Assumption \textbf{(A2)} again, we have
\begin{align*}
I_2
&\le \frac{L_{\tilde{b}}}{N}\,\E\left(\int_{\R^{d}}|X_t^{i,\mu}-z|\,\mu^X_t(dz)\right)
\le \frac{L_{\tilde{b}}}{N}\,\E\left(\int_{\R^{d}}\big(|X_t^{i,\mu}|+|z|\big)\,\mu^X_t(dz)\right)\\
&=\frac{2L_{\tilde{b}}}{N}\int_{\R^{d}}|z|\,\mu^X_t(dz)
\le\frac{2L_{\tilde{b}}}{N}\bigg(\int_{\R^{d}}|z|^2\,\mu^X_t(dz)\bigg)^{1/2},
\end{align*}
where the equality is due to $\int_{\R^{d}}|X_t^{i,\mu}|\,\mu^X_t(dz)=|X_t^{i,\mu}|$ and in the last inequality we used the H\"{o}lder inequality.

Combining with two estimates above yields \eqref{Ai}.
\end{proof}

\begin{lemma}\label{C2}
Suppose that Assumptions {\rm\textbf{(A1)}}-{\rm\textbf{(A3)}}, \eqref{hh} and \eqref{n} hold. Let $(X_t,Y_t)_{t\ge0}$ be the solution to \eqref{Equation} so that $\EE(|X_0|^2+|Y_0|^2)<\infty$. Then
$$\sup_{t\ge0}\EE|X_t|^2\le C_3,$$
where the constant $C_3$ depends on $\ga,\,L_b,\,L_{\tilde{b}},\,\th,\,R_0$, $\EE(|X_0|^2+|Y_0|^2)$ and the second order moment of L\'evy measure $\nu$.
\end{lemma}

\begin{proof}
We partly adopt the idea from the proof of \cite[Lemma 23]{SK}. Recall that $A=\frac{1}{2}(1-2\tau)^2$, $B=(1-2\tau)\ga^{-1}$ and $C=\ga^{-2}$ are given in Proposition \ref{Proposition2}.
By It\^o's formula, Assumptions \textbf{(A1)} and \textbf{(A2)}, it holds
\begin{align*}
&d\big(\langle X_t,AX_t\rangle+\langle X_t,BY_t\rangle+\langle Y_t,CY_t\rangle\big)\\
&=2A\langle X_t,Y_t\rangle\,dt+B\langle Y_t,Y_t\rangle\,dt+B\bigg\langle X_t,b(X_t)+\int_{\R^d}\tilde{b}(X_t,z)\,\hat{\mu}_t^X(dz)-\ga Y_t\bigg\rangle\,dt+B\langle X_t,dL_t\rangle\\
&\quad +2C\bigg\langle Y_t,b(X_t)+\int_{\R^d}\tilde{b}(X_t,z)\,\hat{\mu}_t^X(dz)-\ga Y_t\bigg\rangle\,dt\\
&\quad +C\int_{\R^d}\big(|Y_t+z|^2-|Y_t|^2-2\langle Y_t,z\rangle\I_{\{|z|\le1\}}\big)\nu(dz)dt+dM'_t\\
&\le 2A\langle X_t,Y_t\rangle\,dt+B\langle Y_t,Y_t\rangle\,dt+B\big[-\th |X_t|^2+(L_b+\th )|X_t|^2\I_{\{|X_t|\le R_0\}}\big]dt\\
&\quad +2C L_b|X_t||Y_t|dt+|BX_t+2CY_t||b(0)|dt+|BX_t+2CY_t|\big[L_{\tilde{b}}(\E|X_t|+|X_t|)+|\tilde{b}(0,0)|\big]dt\\
&\quad -\ga B\langle X_t,Y_t\rangle\,dt-2\ga C|Y_t|^2\,dt+\bigg\langle BX_t+2CY_t,\int_{\{|z|>1\}}z\nu(dz)dt\bigg\rangle\\
&\quad +C\int_{\R^d}|z|^2\nu(dz)dt+B\bigg\langle X_t,\int_{\R^d}z\widetilde{N}(dz,dt)\bigg\rangle+dM'_t\\
&\le -\big[(\th B-L^2_b\ga^{-1}C)|X_t|^2+(\ga B-2A)\langle X_t,Y_t\rangle+(\ga C-B)|Y_t|^2\big]\,dt\\
&\quad+|BX_t+2CY_t|\big[L_{\tilde{b}}(\E|X_t|+|X_t|)+|\tilde{b}(0,0)|\big]dt+|BX_t+2CY_t||b(0)|dt\\
&\quad +|BX_t+2CY_t|\int_{\{|z|>1\}}|z|\nu(dz)dt+\bigg(B(L_b+\th)R_0^2+C\int_{\R^d}|z|^2\nu(dz)\bigg)dt\\
&\quad +B\bigg\langle X_t,\int_{\R^d}z\widetilde{N}(dz,dt)\bigg\rangle+dM'_t\\
&=:(J_1+J_2+J_3+J_4)\,dt+\bigg(B(L_b+\th)R_0^2+C\int_{\R^d}|z|^2\nu(dz)\bigg)dt+dM^\ast_t,
\end{align*}
where $(M'_t)_{t\ge0}$ and $(M^\ast_t)_{t\ge0}$ are martingales.

Next, we will estimate $J_i,~i=1,2,3,4.$ Note again that $A=\frac{1}{2}(1-2\tau)^2$, $B=(1-2\tau)\ga^{-1}$ and $C=\ga^{-2}$. First, by \eqref{la} with $\la=2\tau\ga$, we have
\begin{align*}
J_1
&\le -2\tau\ga \bigg(\frac{1}{2}(1-2\tau)^2|X_t|^2+(1-2\tau)\ga^{-1}\langle X_t,Y_t\rangle+\ga^{-2}|Y_t|^2\bigg).
\end{align*}
Recalling that $L_{\tilde{b}}$ satisfies \eqref{n}, we have $L_{\tilde{b}}\le\frac{\tau\ga^2}{16\sqrt{2}}$, and so
\begin{align*}
\E J_2
&\le \frac{\ga^{-1}L_{\tilde{b}}}{2} \bigg(2\sqrt{2}\,\E|(1-2\tau)X_t+2\ga^{-1}Y_t|^2+\sqrt{2}\E|X_t|^2\bigg)+\ga^{-1}\E|(1-2\tau)X_t+2\ga^{-1}Y_t||\tilde{b}(0,0)|\\
&\le \frac{\tau\ga}{16}\E\bigg((1-2\tau)^2|X_t|^2+4(1-2\tau)\ga^{-1}\langle X_t,Y_t\rangle+4\ga^{-2}|Y_t|^2+(1-2\tau)^2|X_t|^2\bigg)\\
&\quad +\frac{\tau\ga}{16}\E|(1-2\tau)X_t+2\ga^{-1}Y_t|^2+\frac{4}{\tau\ga^3}|\tilde{b}(0,0)|^2\\
&= \tau\ga\E\bigg(\frac{3}{16}(1-2\tau)^2|X_t|^2+\frac{1}{2}(1-2\tau)\ga^{-1}\langle Y_t,Y_t\rangle+\frac{1}{2}\ga^{-2}|Y_t|^2\bigg)+\frac{2}{\tau\ga^3}|\tilde{b}(0,0)|^2
\end{align*}
where the first inequality follows from $2ab\le a^2+b^2$, and in the second inequality we used the facts that $(1-2\tau)^2\ge\frac{1}{2}$ with $\tau$ given by \eqref{tau} and $2ab\le a^2+b^2$ as well.\\
Besides, by $2ab\le a^2+b^2$ again, we can obtain
$$
\E J_3\le \tau\ga\E\bigg(\frac{1}{16}(1-2\tau)^2|X_t|^2+\frac{1}{4}(1-2\tau)\ga^{-1}\langle X_t,Y_t\rangle+\frac{1}{4}\ga^{-2}|Y_t|^2\bigg)+\frac{4}{\tau\ga^3}|b(0)|^2
$$
and
$$
\E J_4\le \tau\ga\E\bigg(\frac{1}{16}(1-2\tau)^2|X_t|^2+\frac{1}{4}(1-2\tau)\ga^{-1}\langle X_t,Y_t\rangle+\frac{1}{4}\ga^{-2}|Y_t|^2\bigg)+\frac{4}{\tau\ga^3}\bigg(\int_{\{|z|>1\}}|z|\nu(dz)\bigg)^2.
$$

Combining with all the estimates above yields
\begin{align*}
&\frac{d}{dt}\E\bigg(\frac{1}{2}(1-2\tau)^2|X_t|^2+(1-2\tau)\ga^{-1}\langle X_t,Y_t\rangle+\ga^{-2}|Y_t|^2\bigg)\\
&\le -2\tau\ga\E\bigg(\frac{1}{2}(1-2\tau)^2|X_t|^2+(1-2\tau)\ga^{-1}\langle X_t,Y_t\rangle+\ga^{-2}|Y_t|^2\bigg)\\
&\quad +\tau\ga\E\bigg(\frac{5}{16}(1-2\tau)^2|X_t|^2+(1-2\tau)\ga^{-1}\langle X_t,Y_t\rangle+\ga^{-2}|Y_t|^2\bigg)+\widetilde{C}\\
&\le -\tau\ga\E\bigg(\frac{1}{2}(1-2\tau)^2|X_t|^2+(1-2\tau)\ga^{-1}\langle X_t,Y_t\rangle+\ga^{-2}|Y_t|^2\bigg)+\widetilde{C},
\end{align*}
where $$
\widetilde{C}:=\frac{4}{\tau\ga^3}\bigg[|\tilde{b}(0,0)|^2+|b(0)|^2+\bigg(\int_{\{|z|>1\}}|z|\nu(dz)\bigg)^2\bigg]+(1-2\tau)\ga^{-1}(L_b+\th)R_0^2+\ga^{-2}\int_{\R^d}|z|^2\nu(dz).
$$
By Gr\"{o}nwall's inequality, we obtain
\begin{align*}
&\E\bigg(\frac{1}{2}(1-2\tau)^2|X_t|^2+(1-2\tau)\ga^{-1}\langle X_t,Y_t\rangle+\ga^{-2}|Y_t|^2\bigg)\\
&\le\E\bigg(\frac{1}{2}(1-2\tau)^2|X_0|^2+(1-2\tau)\ga^{-1}\langle X_0,Y_0\rangle+\ga^{-2}|Y_0|^2\bigg)e^{-\tau\ga t}+\frac{\widetilde{C}}{\tau\ga}\\
&\le \max\bigg\{\frac{(1-2\tau)^2}{2}+\frac{1-2\tau}{2\ga},\frac{1}{\ga^2}+\frac{1-2\tau}{2\ga}\bigg\}\E(|X_0|^2+|Y_0|^2)+\frac{\widetilde{C}}{\tau\ga}\\
&=:\widehat{C}<\infty,
\end{align*}
where $\widehat{C}<\infty$ follows from $\E(|X_0|^2+|Y_0|^2)<\infty$ and \eqref{hh}. Note that $\widehat{C}$ depends on $\ga,\,L_b,\,L_{\tilde{b}},\,\th,\,R_0$, $\EE(|X_0|^2+|Y_0|^2)$ and the second order moment of L\'evy measure $\nu$.
Following the third inequality in \eqref{e:add1} with $\eta^2=3/4$, we have
$$
\frac{1}{2}(1-2\tau)^2|X_t|^2+(1-2\tau)\ga^{-1}\langle X_t,Y_t\rangle+\ga^{-2}|Y_t|^2\ge \frac{1}{8}(1-2\tau)^2|X_t|^2,
$$
so the desired assertion follows with $C_3:=\frac{8\widehat{C}}{(1-2\tau)^2}$.
\end{proof}

Finally we give the following
 remark for possible extensions of the main results in our paper.
\begin{remark} In the present setting, we are only concerned on the SDE \eqref{Equation} with the nonlinear term in the drift of the convolution form, and furthermore in Theorem \ref{chaos} we also require that the initial distributions of the associated SDEs admit finite second order moment. Obviously, the arguments of the paper with some modifications can work for the nonlinear term that is much more general rather than of the convolution form. On the other hand, the finite second order moment condition on the initial distributions is due to the proof of Lemma \ref{reAi}, which is partly inspired by that of \cite[Theorem 2]{DEGZ} (see also the proof of \cite[Theorem 1.2]{LMW}). Instead, one may directly apply \cite[Lemma 3.3]{DK} to present an alternative proof for Lemma \ref{reAi} under the $(1+\varepsilon)$-order moment condition on the initial
distributions. Hence, there is room to establish Theorem \ref{chaos} under weaker moment on the initial distributions.
\end{remark}

\section{Appendix}
\subsection{Strong solution to McKean-Vlasov SDEs with L\'evy noises}
Consider the following McKean-Vlasov SDE on $\R^{n}$ with L\'evy noise:
\begin{equation}\label{MV}
dX_t=b(X_t,\scr{L}_{X_t})\,dt+\,\sigma dL_t,
\end{equation}
where  $\scr{L}_{X_t}$ is denoted by the law of $X_t$, $\sigma\in \R^n\times \R^n$, $b:\R^{n}\times\scr{P}_1(\R^{n})\to\R^{n}$
is a measurable map with $\scr{P}_1(\R^n)$ being the class of all probability measures on $\R^n$ with finite first order moment, and $(L_t)_{t\ge0}$ is an $n$-dimensional pure jump L\'evy process so that its L\'evy measure $\nu$ satisfies $\nu(|\cdot|\I_{\{|\cdot|\ge1\}})<\infty.$

\begin{definition}
\begin{itemize}
  \item [{\rm(1)}]A c\`{a}dl\`{a}g adapted process $(X_t)_{t\ge0}$ on $\R^{n}$ is called a strong solution to \eqref{MV}, if
  $$\EE\int_0^t|b(X_s,\scr{L}_{X_s})|\,ds<\infty,\quad t\ge 0,$$ and  $\mathbb{P}$-a.s.
  \begin{equation}\label{MV1}
  X_t=X_0+\int_0^tb(X_s,\scr{L}_{X_s})\,ds+\sigma L_t,\quad t\ge0.
  \end{equation}
  \item [{\rm(2)}]We say that \eqref{MV} has strong $($or pathwise$)$ existence and uniqueness in $\scr{P}_1(\R^n)$, if for any $\scr{F}_0$-measurable random variable $X_0$ with $\mathbb{E}|X_0|<\infty$, there exists a unique strong solution $(X_t)_{t\ge0}$ satisfying \eqref{MV1} and $\mathbb{E}|X_t|<\infty$.
  \item [{\rm(3)}]A couple $(\widetilde{X}_t,\widetilde{L}_t)_{t\ge0}$ is called a weak solution to \eqref{MV}, if $(\widetilde{L}_t)_{t\ge0}$ is an $n$-dimensional pure jump L\'evy process with respect to a complete filtered probability space $(\tilde{\Omega},\{\tilde{\scr{F}_t}\}_{t\ge0},\tilde{\PP})$, and $(\widetilde{X}_t)_{t\ge0}$ solves
      $$\widetilde{X}_t=\widetilde{X}_0+\int_0^tb(\widetilde{X}_s,\scr{L}_{\widetilde{X}_s}|_{\tilde{\PP}})\,ds+\sigma \widetilde{L}_t,\quad t\ge0.$$
  \item [{\rm(4)}]\eqref{MV} is said to have weak uniqueness in $\scr{P}_1(\R^n)$, if for any two weak solutions to the equation \eqref{MV} with common initial distribution in $\scr{P}_1(\R^n)$ are equal in law.
\end{itemize}
\end{definition}

We have the following statement from \cite[Theorem 2.3]{DH}.

\begin{theorem}\label{app} Suppose that $b(\cdot,\cdot)$ is continuous on $\R^{n}\times\scr{P}_1(\R^{n})$, and there exists a constant $K>0$ such that, for all $x_1,\,x_2\in\R^{n}$ and $\mu_1,\mu_2\in\scr{P}_1(\R^{n})$,
\begin{align}\label{Lip}
{\langle b(x_1,\mu_1)-b(x_2,\mu_2),x_1-x_2\rangle}\le K(|x_1-x_2|^2+\bb W_1(\mu_1,\mu_2)|x_1-x_2|)
\end{align}
and, for all $\mu\in \scr{P}_1(\R^{n})$,
\begin{equation}\label{Growth}
|b(0,\mu)|\le K(1+\mu(|\cdot|)).
\end{equation}
Then the equation \eqref{MV} has both strong and weak well-posedness in $\scr{P}_1(\R^{n})$.
\end{theorem}

\subsection{Comparison of two metrics $r_s((x,y),(x',y'))$ and $r_l((x,y),(x',y'))$}
\begin{lemma}\label{lemma1-app} For any $(x,y), (x',y')\in\R^{2d}$, let $r_s((x,y),(x',y')):=\al|x-x'|+|x-x'+\ga^{-1}(y-y')|$, and $r_l((x,y),(x',y'))$ be given by \eqref{rl}. Then, there exists a constant $c_0\ge 1$ such that for all $(x,y), (x',y')\in\R^{2d}$,
$$ c_0^{-1} r_l((x,y),(x',y'))\le r_s((x,y),(x',y')) \le c_0 r_l((x,y),(x',y')).$$ More explicitly, both of them are equivalent to the Euclidean distance on $\R^{2d}$. \end{lemma}

\begin{proof} As shown in the argument of \eqref{e:add1}, it holds that
\begin{align*}r_l((x,y),(x',y'))\ge &
\min\Big\{\frac{1-2\tau}{2\sqrt{2}},\frac{\ga^{-1}}{\sqrt{3}}\Big\}(|x-x'|^2+|y-y'|^2)^{1/2}\\
\ge &\frac{1}{\sqrt{2}}\min\Big\{\frac{1-2\tau}{2\sqrt{2}},\frac{\ga^{-1}}{\sqrt{3}}\Big\}(|x-x'|+|y-y'|).\end{align*}
On the other hand, it is clear that
\begin{align*} r_l((x,y),(x',y')) \le &\left(\frac{1}{2}(1-2\tau)^2|x-x'|^2+(1-2\tau)\ga^{-1}| x-x'||y-y'|+\ga^{-2}|y-y'|^2\right)^{1/2}\\
\le&\left((1-2\tau)^2|x-x'|^2+3\ga^{-2}/2|y-y'|^2\right)^{1/2}\\
\le& \max\big\{1-2\tau,\sqrt{3/2}\ga^{-1}\big\}(|x-x'|+|y-y'|).
\end{align*}

It follows from the definition of $r_s((x,y),(x',y'))$ that
$$r_s((x,y),(x',y'))\le (1+\al)|x-x'|+\ga^{-1}|y-y'|$$ and $r_s((x,y),(x',y'))\ge \al|x-x'|$. Furthermore, when $|y-y'|\ge 2\gamma|x-x'|$, we have
$$r_s((x,y),(x',y'))\ge \ga^{-1}|y-y'|-|x-x'|\ge (2\ga)^{-1}|y-y'|.$$ Therefore, there is a constant $c>0$ so that
$$r_s((x,y),(x',y'))\ge c(|x-x'|+|y-y'|).$$

Putting all the estimates together, we know that  both $r_s((x,y),(x',y'))$ and $r_l((x,y),(x',y'))$ are equivalent to the Euclidean distance on $\R^{2d}$. The proof is completed.
\end{proof}

\ \

\noindent \textbf{Acknowledgements.}
We would like to thank two referees for careful comments and corrections.

\vspace{0.3cm}
\section{Declarations}

\noindent \textbf{Funding}
The research is supported by the National Key R\&D Program of China (2022YFA1006003) and the  NSF of China (Nos.\ 12071076 and 12225104).\\

\noindent \textbf{Conflict of Interest} No potential conflict of interest was reported by the authors.\\

\noindent \textbf{Data Availability} Data sharing not applicable to this article as no datasets were generated or analysed during the current study.\\

\noindent \textbf{Authors' Contributions} All the authors contributed to conducting the research and writing the manuscript. \\


\begin{thebibliography}{99}

\bibitem{BW}
Bao, J. and Wang, J.: Coupling approach for exponential ergodicity of stochastic Hamiltonian systems with L\'evy noises, {\it Stoch. Proc. Appl.}, \textbf{146} (2022), 114--142.



\bibitem{BGM}
Bolley, F., Guillin, A. and Malrieu, F.: Trend to equilibrium and particle approximation for a weakly selfconsistent Vlasov-Fokker-Planck equation, {\it M2AN Math. Model. Numer. Anal.},
\textbf{44} (2010), 867--884.

\bibitem{BD}
Bouchut, F. and Dolbeault, J.: On long time asymptotics of the Vlasov-Fokker-Planck equation and
of the Vlasov-Poisson-Fokker-Planck system with Coulombic and Newtonian potentials, {\it Differential
Integral Equations}, \textbf{8} (1995), 487--514.



\bibitem{CGM}
Cattiaux, P., Guillin, A. and Malrieu, F.: Probabilistic approach for granular media equations in the non uniformly convex case, {\it Probab. Theory Related Fields}, \textbf{140} (2008), 19-40.


\bibitem{Cav}
 Cavallazzi, T.: Quantitative weak propagation of chaos for stable-driven McKean-Vlasov SDEs, arXiv:2212.01079.

\bibitem{DH}
Deng, C.S. and Huang, X.: Harnack inequalities for McKean-Vlasov SDEs driven by subordinate Brownian motions, {\it J. Math. Anal. Appl.}, \textbf{519} (2023), no. 126763.

\bibitem{DK}
Du, K., Jiang, Y. and Li, X.: Sequential propagation of chaos, arXiv:2301.09913.



\bibitem{DEGZ}
Durmus, A., Eberle, A., Guillin, A. and Zimmer, R.: An elementary approach to uniform in time propagation of chaos, {\it Proceedings AMS}, \textbf{148} (2020), 5387--5398.


\bibitem{EGZ}
Eberle, A., Guillin, A. and Zimmer, R.: Couplings and quantitative contraction rates for Langevin dynamics, {\it Ann. Probab.}, \textbf{47} (2019), 1982--2010.

\bibitem{GLM}
Guillin, A., Le Bris, P. and Monmarch\'{e}, P.: Convergence rates for the Vlasov-Fokker-Planck equation and uniform in time propagation of chaos in non convex cases, {\it Electron. J. Probab.}, \textbf{27} (2022), article no. 124, 1--44.

\bibitem{GLWZ}
Guillin, A., Liu, W., Wu, L. and Zhang, C.: The kinetic Fokker-Planck equation with
mean field interaction, {\it J. Math. Pures Appl.}, \textbf{150} (2021), 1--23.


\bibitem{GM}
 Guillin, A. and Monmarch\'{e}, P.: Uniform long-time and propagation of chaos estimates for mean
field kinetic particles in non-convex landscapes, {\it J. Stat. Phys.}, \textbf{185} (2021), paper no.\ 5.

\bibitem{LMW}
Liang, M., Majka, M.B. and Wang, J.: Exponential ergodicity for SDEs and McKean-Vlasov processes with L\'evy noise, {\it Ann. Inst. Henri Poincar\'e Probab. Stat.}, \textbf{57} (2021), 1665--1701.


\bibitem{JMW2}
Jourdain, B., M\'el\'eard, S. and Woyczynski, W.A.: A probabilistic approach for nonlinear equations involving
the fractional Laplacian and a singular operator, {\it Potential Analysis}, \textbf{23} (2005), 55--81.

\bibitem{JMW}
Jourdain, B., M\'el\'eard, S. and Woyczynski, W.A.: Nonlinear SDEs driven by L\'evy processes and related PDEs, {\it ALEA},  \textbf{4} (2008), 1--29.

\bibitem{KM}
Kac, M.: {\it Foundations of Kinetic Theory}, in: Proceedings of the Third Berkeley Symposium on Mathematical Statistics and Probability, 1954--1955, \textbf{3}, University of California Press, Berkeley and Los Angeles, 1956, 171--197.

\bibitem{LW}
Lou, D. and Wang, J.: Refined basic couplings and Wasserstein-type distances for SDEs with L\'evy noises, {\it Stoch. Proc. Appl.}, \textbf{129} (2019), 3129--3173.

\bibitem{MW}
Mann Jr. J. A. and Woyczynski, W.A.: Growing fractal interfaces in the presence of self-similar hopping surface
diffusion, {\it Physica A: Statistical Mechanics and its Applications}, \textbf{291} (2001), 159--183.

\bibitem{MS}
M\'el\'eard, S.: {\it Asymptotic Behaviour of Some Interacting Particle Systems, McKean-Vlasov and Boltzmann Models}, in: Probabilistic Models for Nonlinear Partial Differential Equations (Montecatini Terme, 1995), Lecture Notes in Math., vol. \textbf{1642}, Springer, Berlin, 1996, 42--95.


\bibitem{SK}
Schuh, K.: Global contractivity for langevin dynamics with distribution-dependent forces and uniform in time propagation of chaos, to appear in {\it Ann. Inst. Henri Poincar\'e Probab. Stat.}, see also {\rm arXiv:2206.03082}.


\bibitem{SA}
Sznitman, A.-S.: {\it Topics in Propagation of Chaos}, in: \'{E}cole d'\'Et\'e de Probabilit\'es de Saint-Flour XIX-1989, Lecture Notes in Math., vol. \textbf{1464}, Springer, Berlin, 1991, 165--251.



\end{thebibliography}
\end{document}